\pdfoutput=1

\documentclass[12]{amsart}

\usepackage{tikz}
\usepackage{amsmath}
\usepackage{amsfonts}
\usepackage{rotating}
\usepackage{bbm}
\usepackage[all,cmtip]{xy}
\usepackage{subcaption}

\usepackage{amssymb}
\usepackage[latin1]{inputenc}
\usepackage{graphicx}
\usepackage{dsfont}
\usepackage{color}
\usepackage{hyperref}
\usepackage{enumitem}
\usepackage{color}
 
\usepackage{marginnote}
\renewcommand{\phi}{\varphi}
\newcommand{\C}{{\mathbb{C}}}
\newcommand{\R}{{\mathbb{R}}}

\newcommand{\Z}{{\mathbb{Z}}}
\newcommand{\N}{{\mathbb{N}}} 

  \tikzset{%
  highlight/.style={rectangle,rounded corners,fill=red!15,draw,fill opacity=0.5,thick,inner sep=0pt}
  }

\usetikzlibrary{mindmap,scopes,arrows,shapes,chains,positioning,fit,backgrounds,snakes}
\tikzstyle{every picture}+=[remember picture]
\tikzstyle{decision} = [diamond, draw, fill=blue!20,
    text width=2.5em, text badly centered, node distance=2.5cm, inner sep=0pt]
\tikzstyle{block} = [rectangle, draw, fill=blue!20,
    text width=2em, text centered, node distance=2.5cm, rounded corners, minimum height=2em]
\tikzstyle{line} = [draw, very thick, color=black!50, -latex']
\tikzstyle{cloud} = [draw, ellipse,fill=red!20, node distance=2cm,
    minimum height=2em]
\tikzstyle{round} = [draw, circle, node distance=2cm,
    minimum height=2em]

\usepackage{amssymb}
\newtheorem{Theorem}{Theorem}[section]
\newtheorem{Lemma}[Theorem]{Lemma}

\newtheorem{Proposition}[Theorem]{Proposition}

\newtheorem{Definition}[Theorem]{Definition}
\newtheorem{Remark}[Theorem]{Remark}
\newtheorem{Example}[Theorem]{Example}

\def\Z{{\mathbb Z}}
\def\R{{\mathbb R}}

\def \C{{\mathbb C}}

\def \0{\lambda_{0}}

\newcommand{\relmiddle}[1]{\mathrel{}\middle#1\mathrel{}}

\makeatletter
\protected\def\xvcenter{%
  \hbox\bgroup$\everyvbox{\everyvbox{}\aftergroup\m@th\aftergroup$\aftergroup\egroup}%
  \vcenter
}
\DeclareRobustCommand{\midscript}[1]{
  \mathchoice{\mid@script\scriptstyle{#1}}
    {\mid@script\scriptstyle{#1}}
    {\mid@script\scriptscriptstyle{#1}}
    {\mid@script\scriptscriptstyle{#1}}
}
\newcommand{\mid@script}[2]{
  \vcenter{\hbox{$\m@th#1#2$}}
}

\makeatletter

\begin{document}

\title[Periodic orbits in the restricted three-body problem]{On families of  periodic orbits in the restricted three-body problem}

\author{Seongchan Kim}
 \address{Institut de Math\'ematiques, Universit\'e de Neuch\^atel, Rue Emile-Argand 11, 2000 Neuch\^atel, Switzerland}
 \email {seongchan.kim@unine.ch}
\date{\today}

%\subjclass{53C25, 53C21, 58F17, 35J15}

%\date{04.11.15}

\begin{abstract} Since Poincar\'e, periodic orbits have been one of the most important objects in dynamical systems. However, searching them is in general quite difficult. A common way to find them is to construct families of periodic orbits which start at obvious periodic orbits. On the other hand, given two periodic orbits one might ask if they are connected by an orbit cylinder, i.e., by a one-parameter family of periodic orbits. 

In this article we study  this   question for a certain class of periodic orbits in  the planar circular restricted three-body problem.   Our strategy is to compare the Cieliebak-Frauenfelder-van Koert invariants which are obstructions to the existence of an orbit cylinder. 
 
\end{abstract}

\maketitle
\setcounter{tocdepth}{3}
\tableofcontents

\section{Introduction}\label{secintro}

The \textit{three-body problem} studies the motion of three masses in $\R^3$ according to Newton's law of gravitation. 
This problem is notoriously difficult in general and hence one considers special cases.
One is the \textit{restricted   three-body problem}, in which one sets the mass of one of the three masses equal to zero. We call this massless body   the \textit{satellite}.
The other two masses, which will be referred to as the \textit{Earth} and \textit{Moon}, then move in a plane. 
If we choose barycentric coordinates in this plane, then each of the Earth and Moon   orbits about the center of mass in a conic section.

We now make two further simplifying assumptions. 
First, we assume that the primaries move in circular orbits. 
Scaling their masses to be $1- \mu$ and $\mu$ with $\mu \in (0, 1/2] $, we can then choose coordinates in the plane and the time unit in such a way that
the positions of the Earth and Moon are given by
$$
E(t) = -\mu \bigl( \cos (t), -\sin (t) \bigr), \quad
M(t) = (1-\mu) \bigl( \cos (t), -\sin (t) \bigr) ,
$$
see Figure \ref{fig.EM}.
Second, we assume that the satellite moves in the same plane as the two primaries. 
Writing $q(t) \in \R^2$ for its position and $p(t) \in \R^2$ for its momentum,
the Hamiltonian of the satellite is then
\begin{equation} \label{e:H}
H_t(q,p)  =  \frac{1}{2}|p|^2  - \frac{1-\mu}{|q-E(t)|} - \frac{\mu}{|q-M(t)|}  .
\end{equation}
Note that this Hamiltonian is time-dependent. 
To put ourself into a more geometric situation, we follow Jacobi and pass to a rotating coordinate system $q \to e^{-it}q$.
In this new coordinate system, the positions of the Earth and Moon are then fixed, 
$$
E = (-\mu,0), \quad M = (1-\mu,0),
$$
and the Hamiltonian becomes time-independent,
\begin{equation} \label{e:PCR3B}
H_{\text{PCR3BP}}(q,p) =\frac{1}{2}|p|^2- \frac{1-\mu}{|q-E|} - \frac{\mu}{|q-M|}  + q_1p_2-q_2p_1 ,
\end{equation}
at the cost of the ``rotating term'' $q_1p_2-q_2p_1$, which is due to the centrifugal force and the Coriolis force.
The Hamiltonian system~\eqref{e:PCR3B} is called the \textit{planar circular restricted three-body problem},  the PCR3BP for short.
Note that stationary points of the PCR3BP correspond to special periodic (in fact, circular) orbits of~\eqref{e:H}.

\begin{figure}[t]
\begin{center}
    \begin{tikzpicture} 
	\node(E)[draw,circle, minimum size = 3.5cm] at (5,2.5) { };
	\node(E)[draw,circle, minimum size = 1.5cm] at (5,2.5) { };
	\draw[dashed] (4.5,0.8 ) -- (5.5, 4.2);
	\draw[line, black, thin] (2,2.5) -- (8, 2.5);
	\draw[line, black, thin] (5,0) -- (5, 5);
\draw[fill] (4.8,1.8) circle [radius=1.2mm];
\draw[fill] (5.5,4.2) circle [radius=1.2mm];
	\draw (4.3,1.8 ) node[below]{$E(t)$} ;
	\draw (6, 4.8) node[below]{$M(t)$} ;
	\draw[line, black](4.8,1.8 ) -- (3.8,2.2 );
	\draw[line, black] (5.5, 4.2)  -- (6.5,3.8 );
	\draw (5.9,2.5) node[below]{$\mu$} ;
	\draw (6.9,2.53) node[below]{$1-\mu$} ;
   \end{tikzpicture}
\end{center}
 \caption{The motion of the two primaries in the PCR3BP before
passing to a rotating coordinate system.}
 \label{fig.EM}
\end{figure}
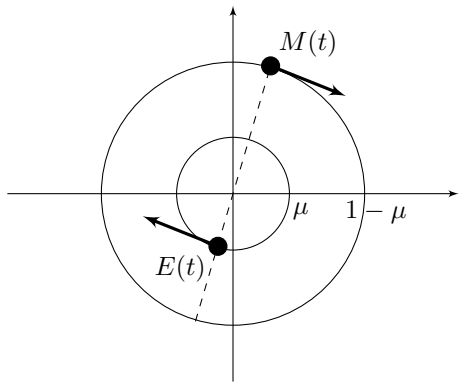

Even though we have made many simplifying assumptions, the dynamics in PCR3BP is still
extremely complicated, whence it generates still many unexplored questions.
One of them is the existence of periodic orbits, 
a problem studied by many outstanding mathematicians over the last two centuries. 
(We refer to \cite{Bruno} for historical informations.) 
The search for periodic orbits of the PCR3BP was adverted by
Poincar\'e in his beautiful book~\cite{Poin}:
\begin{quote}
``What renders these periodic solutions so precious is that they are, so to speak, 
the only breach through which we may try to penetrate a stronghold previously reputed to be impregnable."
\end{quote}

\noindent His strategy to find periodic orbits in the PCR3BP is looking at the family of Hamiltonians
\begin{equation} \label{e:Hs}
H^{\mu}(q,p)  = \frac{1}{2}|p|^2 - \frac{1-\mu}{|q-E|} - \frac{\mu}{|q-M|}  + q_1p_2-q_2p_1  
\end{equation}
in which the parameter   is the mass ratio of the two primaries. 
For $\mu \neq 0$, we have the PCR3BP, while 
\begin{equation} \label{e:HK}
H_{\text{RKP}}(q,p):=H^0(q,p) \,=\, \frac 12 p^2 - \frac{1}{|q|}  + q_1p_2-q_2p_1 
\end{equation}
is  the \textit{rotating Kepler problem}, namely the usual Kepler problem written in our rotating coordinates. 
Since periodic orbits in the rotating Kepler problem are easy to find,  we may hope to find an "orbit cylinder" $\gamma^{\mu}$ at least for $\mu$ in a small interval around $0$, where each $\gamma^{\mu}$ is a periodic orbits of $H^{\mu}$.  The implicit function  theorem indeed shows that 
there is an orbit cylinder $\gamma^{\mu}$, $\mu \in [0,\epsilon)$, emanating from the (rotating) Kepler ellipse
$\gamma^0$, so that one finds a periodic orbit of the PCR3BP, provided that $\mu$ is small enough,  see for example \cite{Arenstorf, Barrar}.

 In this paper,  we take a closer look at \textit{Euler's problem of two fixed centers}, which   was first introduced by Euler \cite{Euler1, Euler2}.
  He considered this problem as a starting point to study the PCR3BP.
 The describing Hamiltonian   is obtained from the PCR3BP by forgetting the rotating term, 
\begin{equation*} \label{e:HE}
H_{\text{Euler}}(q,p) = \frac{1}{2}|p|^2 - \frac{1-\mu}{|q-E|} 
- \frac{\mu}{|q-M|} .
\end{equation*}
This system describes the dynamics of the satellite
attracted by two masses of mass $1-\mu$ and $ \mu$ that are fixed at $E$ and~$M$.
If $\mu = 1/2$ and $\mu$ is viewed as a charge instead of a mass, then 
this system can also be seen as describing the motion of an electron attracted
by two protons, as in the hydrogen molecule, see \cite{pauli}.
 
\begin{figure}[h]
\begin{center}
\begin{tikzpicture} 
 \node at (14,6) [left]{The PCR3BP};
 \node at (16.8,5.2) [left]{$H_{\text{PCR3BP}} = \frac{1}{2}|p|^2 - \frac{1-\mu}{|q-E|} - \frac{\mu}{|q-M|} + q_1 p_2 - q_2 p_1$};
 
\node at (12,2.5) [left]{The rotating Kepler problem};
 \node at (12.2,1.7) [left]{$H_{\text{RKP}} = \frac{1}{2}|p|^2 - \frac{1}{|q|}  + q_1 p_2 - q_2 p_1$};
 \node at (11.2, 4) [left]{$\mu \rightarrow 0$};
\draw [->](12.5,4.5) to  (10,3);
 
\node at (18,2.5) [left]{The Euler problem};
 \node at (18.95,1.7) [left]{$H_{\text{Euler}} = \frac{1}{2}|p|^2 - \frac{1-\mu}{|q-E|} - \frac{\mu}{|q-M|}$};
  \node at (17.5,4) [left]{$q_1p_2 - q_2 p_1 \rightarrow 0$};
\draw [-> ](13.5, 4.5) to  (16,3);

\end{tikzpicture}
 \end{center}
 \caption{Two special cases of the PCR3BP}
\label{figure1}
\end{figure}
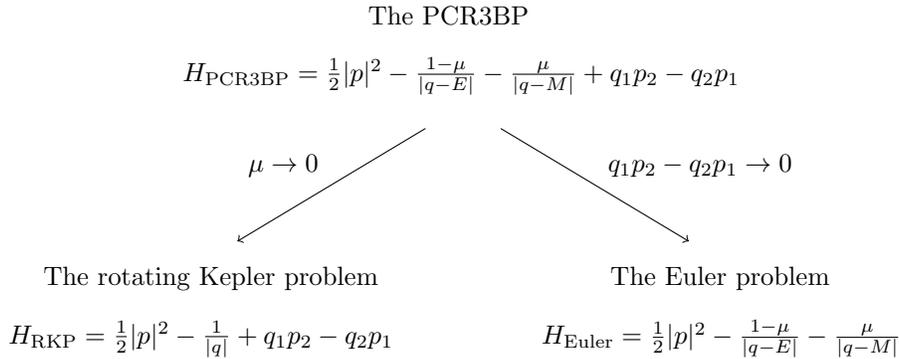

 Applying Poincar\'e's strategy to Euler's idea naturally leads to the question: \textit{Can one find periodic orbits in the PCR3BP by homotoping periodic orbits in the Euler problem? If it is possible, how different are they from the ones  obtained from periodic orbits in the rotating Kepler problem?} In this paper, we study this question by  computing the invariants $\mathcal{J}_1$  and~$\mathcal{J}_2$, 
which recently introduced by
Cieliebak--Frauenfelder--van Koert in~\cite{invariant}, for periodic orbits in the Euler problem.
For definition and properties of the invariants, see Section 2.
These invariants serve as obstructions to the existence of a homotopy of periodic orbits. 
We then   compare them with the ones for the rotating Kepler problem provided in \cite{Kim}.

\medskip \noindent 
\textbf{Theorem.} \textit{Let $\gamma^{\text{RKP}}$ and $\alpha^{\text{Euler}}$ be periodic orbits in the rotating Kepler problem and Euler problem, respectively, which are freely homotopic. 
Then the Cieliebak-Frauenfelder-van Koert invariants   agree on $\gamma^{\text{RKP}}$ and~$\alpha^{\text{Euler}}$.}

 \bigskip
 \noindent
\textit{Remark.} Recall that the Conley-Zehnder indices of periodic orbits in the rotating Kepler problem and Euler problem agree \cite{RKP, Kim2}. In view of this and   the previous theorem we expect that there might be  more intrinsic relationship between the two problems. 
In other words, there might be a hope that during a   homotopy of  periodic orbits from  the Euler problem via the PCR3BP to the rotating Kepler problem, no birth-death bifurcation occurs, see Figure \ref{sdf34}. 
%Indeed, if there is no homotopy of periodic orbits between the two problems, then there is no reason why the Cieliebak-Frauenfelder-van Koert invariants agree. 

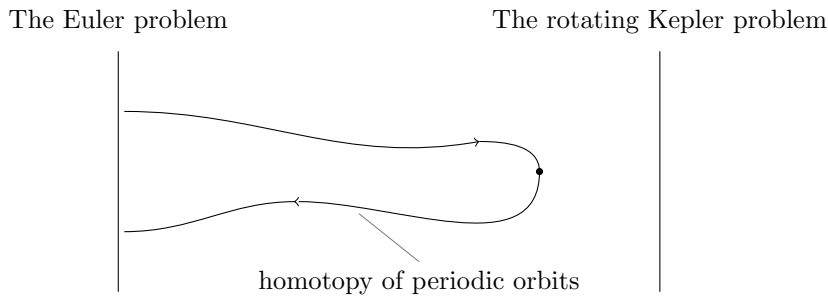
\begin{figure}[h]
\begin{center}
 \begin{tikzpicture}[scale=0.8]

\draw (-5,-2) to (-5,2);
\draw ( 4,-2) to (4 ,2);
\node at (-5, 2.5) {The Euler problem};
\node at (4,2.5) {The rotating Kepler problem};
%\draw[ ->] (-4.3, -2.5) to (3.3, -2.5) ;
%\node at (-0.5, -3) {homotopy};
\draw[->] (-4.9, 1) to [out=0, in= 190](1, 0.5);
\draw  ( 1, 0.5) to [in=90, out= 0](2, 0);
 
\draw (2,0) to [out=270, in=0] (-2, -0.5);
\draw[ >-] (-2,-0.5) to [out=180, in=0] (-4.9, -1);
  \draw[fill] (2,0) circle [radius=0.05];
\draw[gray] (-1, -0.7) to (-0, -1.5);
\node at (0, -1.5) [below] {homotopy of periodic orbits};

\end{tikzpicture}
 \end{center}
 \caption{A possible birth-death bifurcation. Due to the result of the theorem, we expect that this does not happen.}
\label{sdf34}
\end{figure}

\noindent
\textit{Remark.} The $\mathcal{J}_1, \mathcal{J}_2$ invariants are defined for periodic orbits near either $E$ or $M$. 
In \cite{CFZ19}, Cieliebak-Frauenfelder-Zhao define new invariants for periodic orbits which encircle both primaries. 
An explicit example for such invariant are provided in \cite{Kim19a}.

\bigskip

\noindent
\textbf{Acknowledgements:} The author would like to thank    his supervisor   Urs Frauenfelder for interesting him in this subject. He is  also grateful  to Felix Schlenk and Holger Waalkens for valuable discussions, to    the unknown referee for valuable comments   and to   the Institute for Mathematics of University of Augsburg for providing a supportive research environment. This research was supported  by DFG  grants CI 45/8-1 and FR 2637/2-1. The results of this paper will form part of the author's PhD thesis.

\section{$J^+$-like invariants for planar closed curves}\label{secinvariants}

In this section we recall some results on periodic orbits in planar Stark-Zeeman systems \cite{invariant} and   generalize slightly them. 

\subsection{Planar Stark-Zeeman systems} \label{szstyemsd}

Let $U \subset (\R^2, q_1, q_2)$ be an open neighborhood of the origin whose closure is diffeomorphic to the unit closed disk.  The standard symplectic form $\omega_0$ on $(T^* U, q_1, q_2, p_1, p_2)$ is given by $dp_1 \wedge dq_1 + dp_2 \wedge dq_2$. Fix a 2-form  $\sigma_B := B(q) dq_1 \wedge dq_2 \in \Omega^2( U)$ which is called a \textit{magnetic form}.  Since the second de Rham cohomology group of $U$ is trivial, the 2-form $\sigma_B$ is exact, i.e., one finds a 1-form $\alpha_A=A_1(q) dq_1 + A_2(q) dq_2 \in \Omega^1(U)$ such that $d\alpha_A = \sigma_B$. Note that $A_1$ and $A_2$ satisfy
$$
B = \frac{\partial A_2}{\partial q_1} - \frac{\partial A_1}{\partial q_2}.
$$
The \textit{twisted symplectic form} on $T^*U$ is defined to be
$$
\omega_B := \omega_0 + \pi^* \sigma_B,
$$
where $\pi :T^* U \rightarrow U$ is the footpoint projection. Note that if the smooth function $B=B(q)$ is identically zero, then the twisted symplectic form equals the standard symplectic form.

Let $V_1:U \rightarrow \R$ be a smooth function and consider the Hamiltonian
$$
H:T^*(U \setminus \left \{(0,0)\right\}) \rightarrow \R,\quad H(q,p)=\frac{1}{2}|p|^2 - \frac{1}{|q|} +V_1(q).
$$
Define the diffeomorphism $\Phi_A : (T^*(U \setminus \left \{(0,0)\right\}), \omega_0) \rightarrow (T^*(U \setminus \left \{(0,0)\right\}), \omega_B)$  by
$$
\Phi_A(q,p):= (q,p-A(q) ),\quad \quad A(q)=(A_1(q), A_2(q) ). 
$$
One can easily see that the map $\Phi$ is in fact a symplectomorphism which implies that $\Phi$ transforms the Hamiltonian equations of $H$ with respect to the twisted symplectic form $\omega_B$ into the Hamiltonian equations of 
\begin{equation}\label{hamiltoniansz}
H_A (q,p):= \Phi_A^* H(q,p) = \frac{1}{2}|p-A(q)|^2  - \frac{1}{|q|} +V_1(q)
\end{equation}
with respect to the standard symplectic form $\omega_0$.

\begin{Definition} \rm (Cieliebak-Frauenfelder-van Koert, \cite{invariant}) A \textit{planar Stark-Zeeman system} is a Hamiltonian system associated to a Hamiltonian of the form  \eqref{hamiltoniansz} defined on $(T^*(U \setminus \left \{(0,0)\right\}), \omega_0)$.  
\end{Definition}

In other words, a planar Stark-Zeeman system describes the dynamics of a particle which moves along integral curves of the Hamiltonian flow $X_{H_A}$ which is defined by $\iota_{X_{H_A}} \omega_0 = -dH_A$.

\begin{Remark} \label{remarksteidhi} \rm  A planar Stark-Zeeman system with $A \equiv 0$ is often called a  \textit{planar Stark system}. 
\end{Remark}

\subsection{Periodic orbits in Stark-Zeeman systems}\label{secperSZ} In the following we fix a Stark-Zeeman system $(\sigma_B, \alpha_A , H:=H_A)$.  Let $c_1 \in \R \cup \left \{ \infty \right \}$ be the energy level such that for any $c<c_1$ the following are satisfied:
\begin{itemize}
\item $c$ is a regular value of the \textit{effective potential} $V(q):= - \tfrac{1}{|q|} +V_1(q)$; and
\item the \textit{Hill's region} $\mathcal{K}_c := \left \{ q \in U \; | \; V(q) \leq c \right \}$ contains a unique bounded component $\mathcal{K}_c^b$ whose closure is diffeomorphic to the unit closed disk.
\end{itemize}
Consider  a family of periodic orbits $q^s : S^1 \rightarrow \mathcal{K}_c^b$, $s \in (-\epsilon, \epsilon)$ in a planar Stark-Zeeman system of   energies less than $c_1$. We assume that  the orbits are allowed to pass through the origin. Indeed, this is always  possible by regularizing the system, see \cite[Section 3]{invariant}. Consider a point $q_0 = q^0(t_0)$.  We distinguish the following three cases:

\;\;

\textit{Case 1.} $q_0 \in \partial \mathcal{K}_c^b$.\\
In this case, the point $q_0$ satisfies $V(q_0)=c$ and then $p=A(q)$. In particular, in view of the Hamiltonian equations, the velocity vanishes at $t=t_0$: $\dot{q}^0 (t_0) =0$. We further assume that the component of $\dot{q}^s(t_0)$ which is normal to $\nabla V( q^s(t_0) )$ changes sign through the point $q_0$.

\begin{enumerate}[label=(\roman*)]
 \item if $B(q_0) \neq 0$, then   by \cite[Lemma 1]{invariant}, we see that the orbit $q^0$ has a cusp at $t=t_0$ and during the family $q^s$, birth or death of an exterior loop happens through the cusp, see Figure \ref{cuspinfty};
\item if $B(q_0) =0$, then imitating the proof of \cite[Lemma 1]{invariant}, we prove
\end{enumerate}
\begin{Lemma}  \label{lemmabakc1}  If the orbit $q^0$ touches the boundary $\partial \mathcal{K}_c^b$, then it bounces back from the boundary, see Figure \ref{lemma25}.
\begin{figure}[h]
 \centering
 \includegraphics[width=1.0\textwidth, clip]{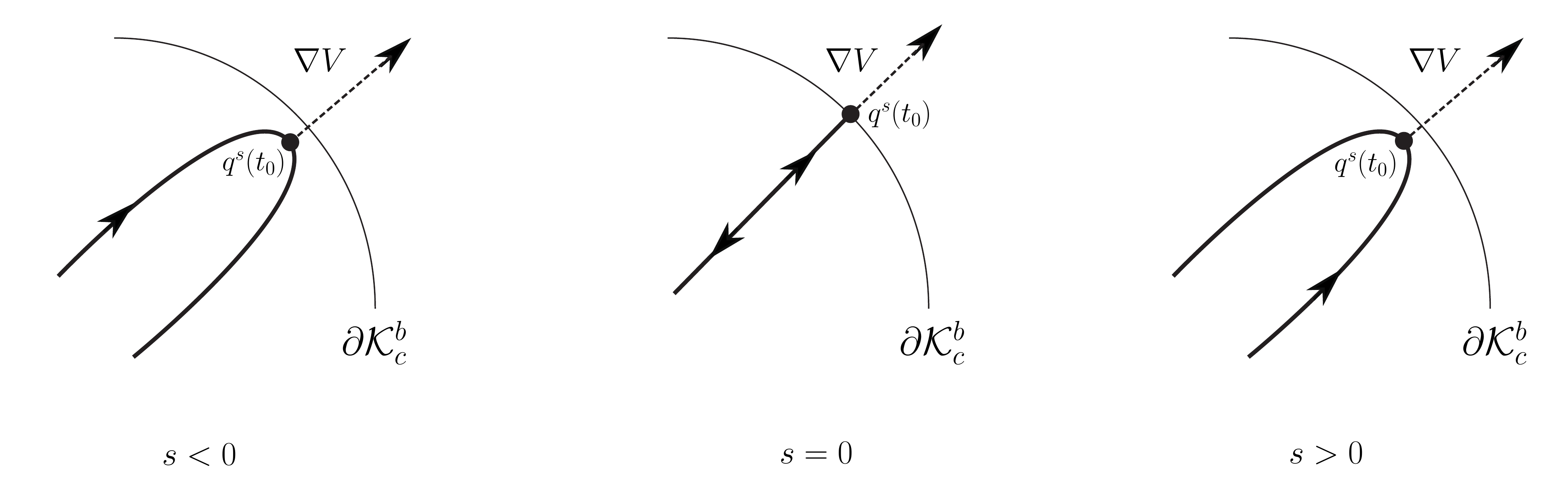}
 \caption{Passing through a touch the boundary of the  Hill's region for the case $B(q_0)=0$}
 \label{lemma25}
\end{figure}
\end{Lemma} 
\begin{proof} The proof is nothing but a rephrase of that of \cite[Lemma 1]{invariant}. Without loss of generality, we may assume that $t_0=0$. Consider an orbit $ q^{s_0}$ in a family $q^s$ near $q^0$. In view of the Hamiltonian equations we have
\begin{equation*}
\ddot{q}^{s_0} (0)=- \frac{\partial V}{\partial q} (q^{s_0} (0)) \quad \text{ and } \quad  \dddot{q}^{s_0}(0) = -D^2 V(q^{s_0}(0)) \dot{q}^{s_0}(0).
\end{equation*}
Since $c$ is a regular value of the effective potential $V$, the second derivative $\ddot{q}^{s_0}(0)$ does not vanish. Moreover, since $q^{s_0}$ is close enough to $q^0$,  we see that $\dddot{q}^{s_0}(0)$ is very small. Choose complex coordinates in which $q^{s_0}(0)=0$, $\dot{q}^{s_0} (0) =a+ib$ and $\ddot{q}^{s_0} (0)= 2$. The Taylor expansion of $q^{s_0}$ in these coordinates   is then given by
\begin{equation*}
q^{s_0} (t)  =  q^{s_0} (0) + \dot{q}^{s_0} (0) t + \frac{1}{2} \ddot{q}^{s_0} (0) t^2  + O(t^3) =  (at+t^2) +ibt  + O(t^3).
\end{equation*}
 Writing $z=x+iy$ and ignoring the terms whose orders are higher than 2, we obtain
\begin{equation*} 
q^{s_0} ~:~ x =  \frac{1}{b^2} \bigg(       y + \frac{ab}{2}     \bigg)^2  - \frac{a^2}{4} 
\end{equation*}
for $b \neq 0$ and $q^{s_0}(t) = ( at+t^2, 0)$ for $b=0$. This shows that for  $b \neq 0$, the orbit $q^{s_0}$ is  a parabola and its width becomes narrower as $b$ tends to zero. Finally, the width becomes zero which means that  the orbit $q^0$ is the ray $\theta=0$  in the chosen coordinates. Moreover, the assumption that the component of $\dot{q}^s(0)$ normal to $(\partial V/\partial q)( q^{s}(0))$ changes sign through $q_0$ implies that $b$ changes sign. Therefore, the orientation of the orbit changes through the point $q_0$. This phenomenon persists under higher order perturbations.  This completes the proof of the lemma.
\end{proof}

\;\;

\textit{Case 2.} $q_0 =(0,0)$.\\
The orbit $q^0$ collides with the origin at time $t=t_0$.
\begin{enumerate}[label=(\roman*)]
 \item if $B(q_0) \neq 0$, then  by \cite[Lemma 2]{invariant} the orbit $q^0$ has a cusp at $t=t_0$ and birth or death of a  loop around the origin happens through the cusp, see Figure  \ref{cusporigin};
\item if $B(q_0) =0$, then  it is obvious that  the orbit $q^0$  bounces back from the origin, see Figure \ref{lemma24}.
\begin{figure}[h]
 \centering
 \includegraphics[width=1.0\textwidth, clip]{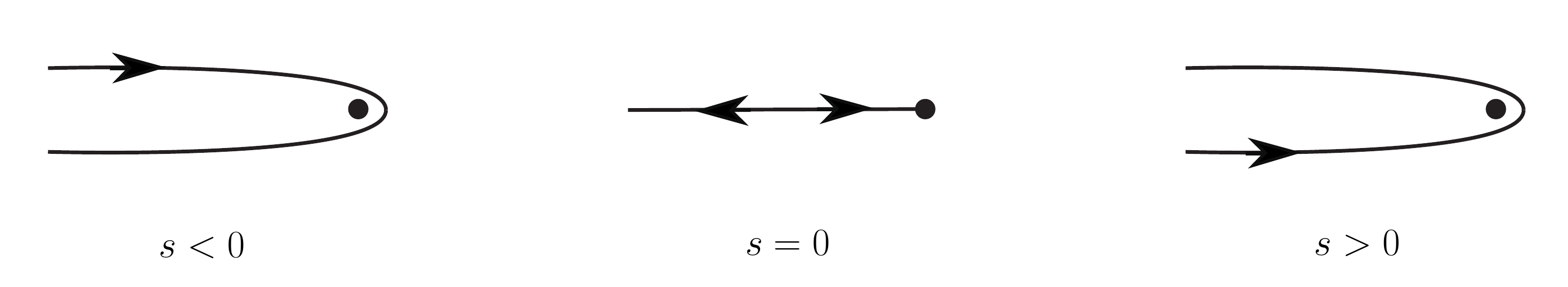}
 \caption{Passing through the origin}
 \label{lemma24}
\end{figure}   
\end{enumerate}

\textit{Case 3.} $q_0 \neq 0$ lies in the interior of $\mathcal{K}_c^b$.\\
In both cases that $B(q_0) \neq 0$ and $B(q_0) =0$,   the orbit $q^0$ is an immersion near $t=t_0$.

\bigskip

We now discuss the case $B \equiv 0$ in more details.  Let $q^s$ be a family of periodic orbits in Stark systems. The previous arguments show that an additional loop will not be attached to $q^s$.  Moreover, we see that the absolute value of  the  winding number of periodic orbits around the origin does not change during the family, but  the  sign  might change.

\begin{Lemma}\label{lemmainverse} Assume that the orbit $q=q^0$ admits an  inverse self-tangency at one point, say $q(t_0)=q(t_1)$, $t_0 \neq t_1$. Then every point on the orbit $q$ is an inverse self-tangency. Moreover, either $q$  touches the boundary of the Hill's region or  it collides with the origin.
\end{Lemma}
\begin{proof} Recall that the Hamiltonian of a Stark system   has the form
$$
H(q,p) = \frac{1}{2} |p|^2 + V(q), \quad \quad V(q) := - \frac{1}{|q|} + V_1(q).
$$
Since a Stark system admits the anti-symplectic involution $(q,p) \mapsto (q,-p)$, it is invariant under   time reversal.  In view of the hypothesis of the lemma and the uniqueness theorem of O.D.E., the orbit $q$ coincides with $\overline{q}$, where $\overline{q}(t):=q(-t)$, up to a suitable time translation. This is the case only if   all points on $q$ are inverse self-tangencies except for two points at which the orbit $q$  bounces back. In view of the arguments in Cases 1 and 2 it follows that each of  these two points either touches the boundary of the Hill's region or collides with the origin. This finishes the proof of the lemma.
\end{proof}
\noindent In view of the previous lemma we define
\begin{Definition}\label{deforbits} \rm A periodic orbit in a planar Stark-Zeeman system with $B\equiv 0$ which has inverse self-tangencies is called 
\begin{itemize}
\item a \textit{brake-brake orbit} if the satellite touches the boundary of the Hill's region at two (distinct) points;
\item a \textit{collision-collision orbit} if the satellite collides with the origin twice (with distinct momenta); and
\item a \textit{brake-collision orbit} if the satellite touches the boundary of the Hill's region and also collides with the origin, see Figure \ref{lemddmaddd2dddddd4ddd}.
\end{itemize}
\begin{figure}
 \centering
 \includegraphics[width=1.0\textwidth, clip]{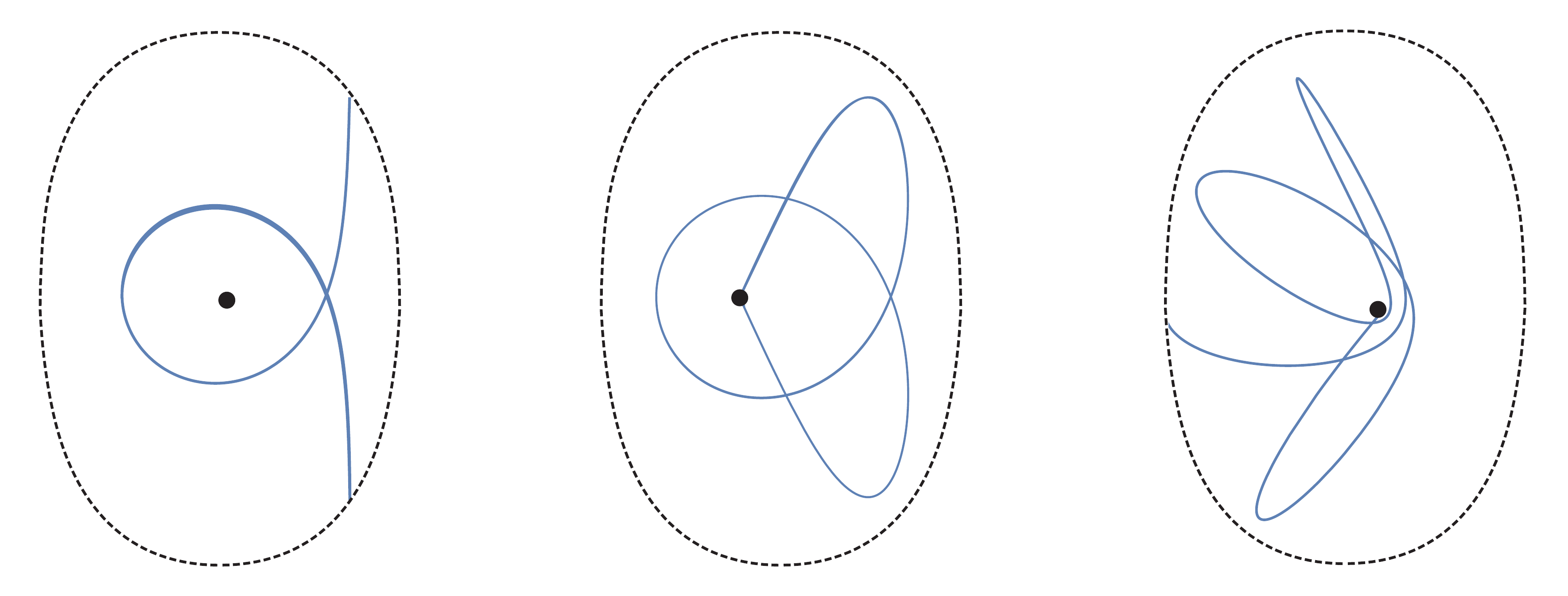}
 \caption{A brake-brake orbit (left), a collision-collision orbit (middle), and a brake-collision orbit (right)}
 \label{lemddmaddd2dddddd4ddd}
\end{figure}  
\end{Definition}

\subsection{Invariants for Stark-Zeeman homotopies}

In the previous section we defined the notion of generic families of periodic orbits in Stark-Zeeman systems. We now introduce invariants for such families.

\subsubsection{The Whitney-Graustein theorem} \index{Whitney-Graustein theorem}

 Let $\gamma:S^1 \rightarrow \C$  be an immersion.   By abuse of notation  we also call the image $K:= \text{im}(\gamma)$ of $\gamma$ an immersion.     The \textit{rotation number} $\text{rot}(\gamma)$ of $\gamma$ is defined to be the degree of the map
$$
S^1 \longrightarrow S^1 , \quad \quad t \mapsto \frac{ \gamma'(t)}{| \gamma'(t)|}.
$$
A one-parameter family of immersions is called a \textit{regular homotopy}. Note that the rotation number is invariant under regular homotopies.

\begin{Theorem} { \rm (Whitney-Graustein, \cite{Whitney})} There exists a bijection between regular homotopy classes of immersions from $S^1$ to $\C$ and the set of integers, where the correspondence is given by
$[ \gamma ] \mapsto \text{rot}(\gamma)$.
\end{Theorem}

\subsubsection{Arnold's $J^+$-invariant} \index{Arnold's $J^+$-invariant}

By a \textit{generic immersion}, we mean an immersion only with transverse double points. Note that generically a homotopy through generic immersions  admits three disasters: triple intersections and direct and inverse self-tangencies. Recall that a  self-tangency is said to be \textit{direct} and \textit{inverse} if the two tangent vectors at the intersection point have the same direction and the opposite directions, respectively. A regular homotopy   $(K^s)_{s\in[0,1]}$ is called a \textit{generic homotopy} if  each $K^s$ is a generic immersion except at  finitely many $s\in (0,1)$ at which  $K^s$ admits crossings through either a triple point or a self-tangency.   

 In \cite{Arnold}, Arnold introduce three invariants for generic homotopies without the three disasters. Among them, we are interested in the $J^+$-invariant which does not change under 
\begin{itemize}
    \item [$(II^+)$] crossing through an inverse self-tangency, see Figure \ref{inverse}; and \index{disaster!self-tangency}
   \item  [$(III)$] crossing through a triple point, see Figure \ref{triple}. \index{disaster!triple point}
\end{itemize}
However, it behaves sensitively to direct self-tangencies: under a positive (or a negative) crossing through a direct self-tangency, which increases (or decreases)  the number of  double points, $J^+$  is increased (or decreased) by two. Different from the rotation number, the $J^+$-invariant does not depend on the orientation. We normalize it by
\begin{equation*}
J^+(K_j) = \begin{cases} 2-2j &  j \neq 0,  \\ 0 & j=0, \end{cases}
\end{equation*}
where  $K_j$, $j \in \N \cup \left \{ 0 \right \}$,  are the \textit{standard curves}: $K_0$ is the figure eight, $K_1$ is the circle, and for each  $j \geq 2$ the curve $K_j$ is given as in the following figure. \index{Arnold's $J^+$-invariant!standard curves}
 \begin{figure}[h]
 \centering
 \includegraphics[width=0.4\textwidth, clip]{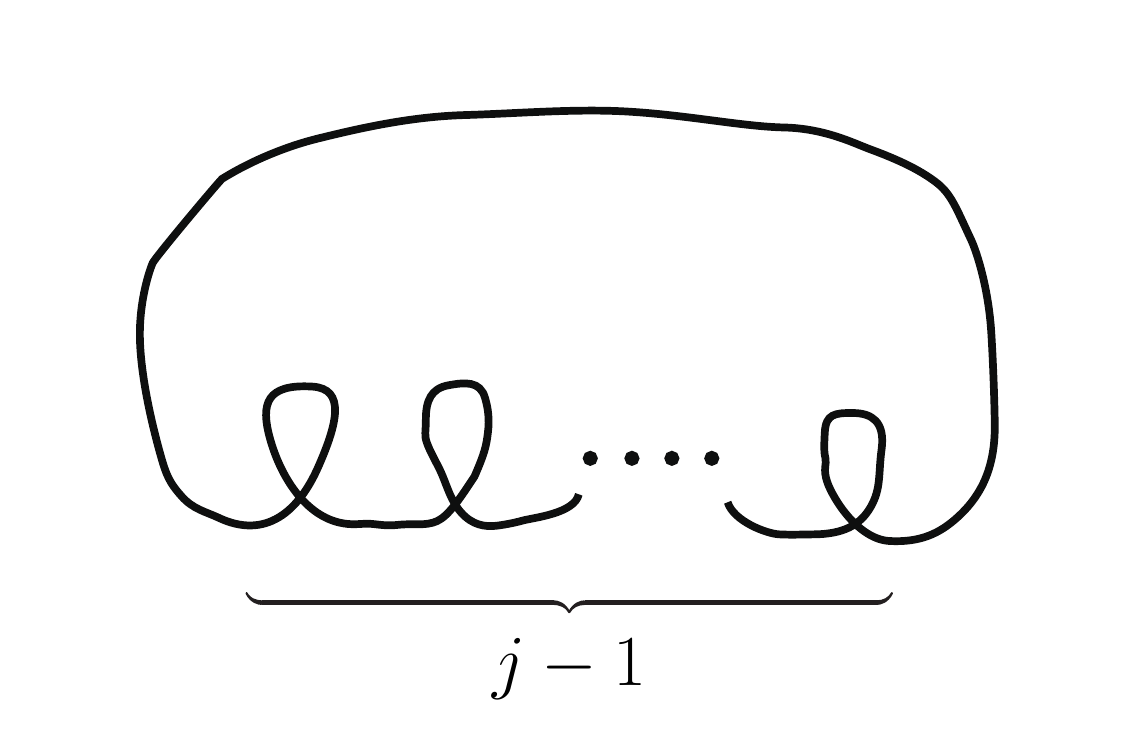}
 \label{standard}
\end{figure}
Let $K$ be a generic immersion. We assume that it is generically homotoped to $K_j$ having $N_1$ positive crossings and $N_2$ negative crossings through a direct self-tangency. Then the abovementioned rules imply
$$ 
J^+ (K) = J^+(K_j) - 2N_1 + 2N_2.
$$

\begin{Example}  In the following  we homotope $K$ to a circle. During the homotopy, we meet one negative crossing and one positive crossing through a direct self-tangency. Since the $J^+$-invariant of a circle equals zero, this shows that $J^+(K)=0$.
\begin{figure}[h]
  \centering
  \includegraphics[width=1.0\linewidth]{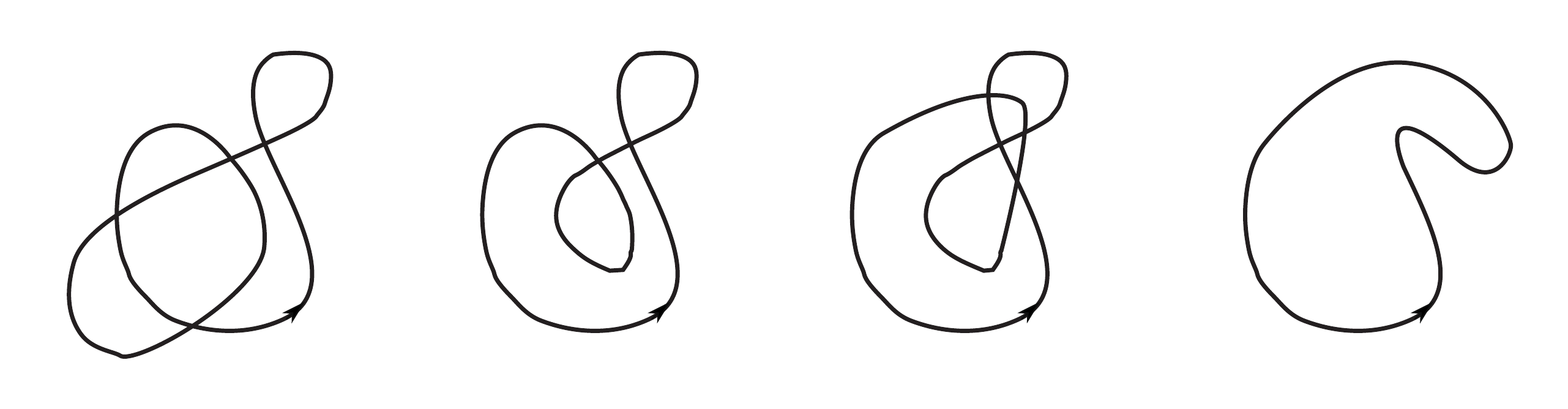}
\end{figure}
\end{Example}

It is worth pointing out that $J^+ \in 2 \Z.$ Moreover, the $J^+$-invariant is additive under connected sums, see \cite[Chapter 1]{Arnold}. In particular, in view of  $J^+(K_1) =0$ we see that the $J^+$-invariant does not change under being attached by additional loops in the unbounded component of the complement $\C \setminus K$.

\subsubsection{The Cieliebak-Frauenfelder-van Koert invariants}

Recall from Section \ref{secperSZ} that during a family of periodic orbits in a Stark-Zeeman system (with $B\neq 0$) the following disasters can happen:
\begin{itemize}
    \item [$(I_{\infty})$]  birth or death of exterior loops through cusps at the boundary of (the bounded component  of) the Hill's region, see Figure \ref{cuspinfty}; and
    \item [$(I_0)$] birth or death of loops around the origin through cusps at the origin, see Figure \ref{cusporigin}. 
\end{itemize}
Meanwhile in a Stark system these two disasters never happen and    the event $(II^+)$ is equivalent to the appearance of one of the three distinguished periodic orbits in  Definition \ref{deforbits}. 

These observations  show that  families of periodic orbits in Stark-Zeeman systems are not generic homotopies. This  led Cieliebak, Frauenfelder and van Koert to introduce  the notion of   Stark-Zeeman homotopies which represent generic 1-parameter families of (simple covered) periodic orbits in (varying) planar Stark-Zeeman systems. 

\begin{Definition}\label{defSZ}  \rm  (\cite[Definition 1]{invariant}) A 1-parameter family $(K^s)_{s\in [0,1]}$ of closed curves in $\C$ is called a \textit{Stark-Zeeman homotopy} if each member is a generic immersion  in $\C \setminus \left \{ (0,0)\right\}$ except for the disasters    $(I_{\infty})$, $(I_0)$, $(II^+)$, and $(III)$ at finitely many $s\in (0,1)$, see  Figure \ref{dislihlsijstrakyeihsed}.  \index{homotopy!Stark-Zeeman}
\end{Definition}
\begin{figure}[h]
\begin{subfigure}{0.9\textwidth}
  \centering
  \includegraphics[width=0.78\linewidth]{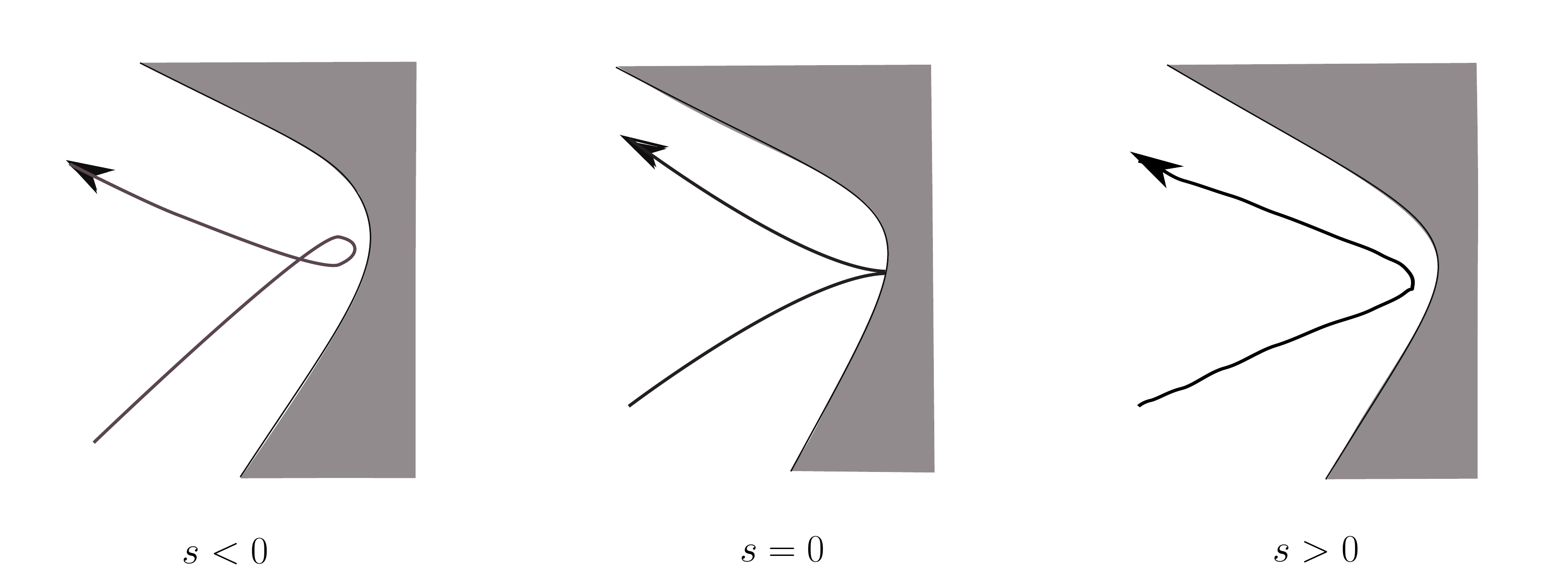}
\caption{A death of an exterior loop   through a cusp}
  \label{cuspinfty}
\end{subfigure}
\begin{subfigure}{0.9\textwidth}
  \centering
  \includegraphics[width=0.8\linewidth]{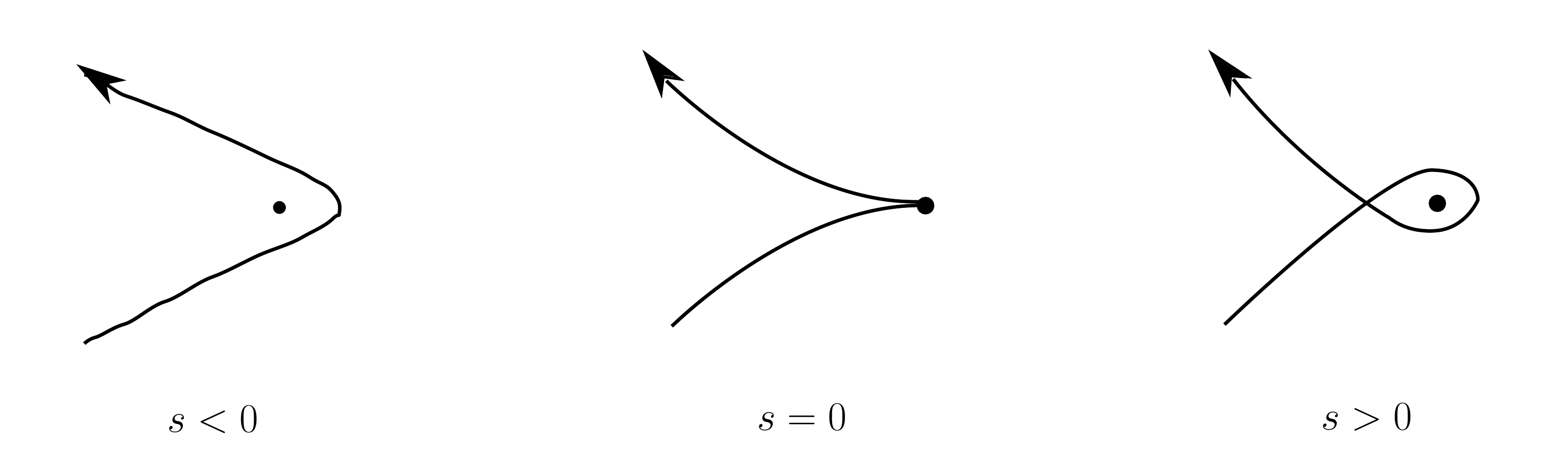}
\caption{A birth of a loop around the origin through a cusp}
  \label{cusporigin}
\end{subfigure}
\begin{subfigure}{0.9\textwidth}
  \centering
  \includegraphics[width=0.8\linewidth]{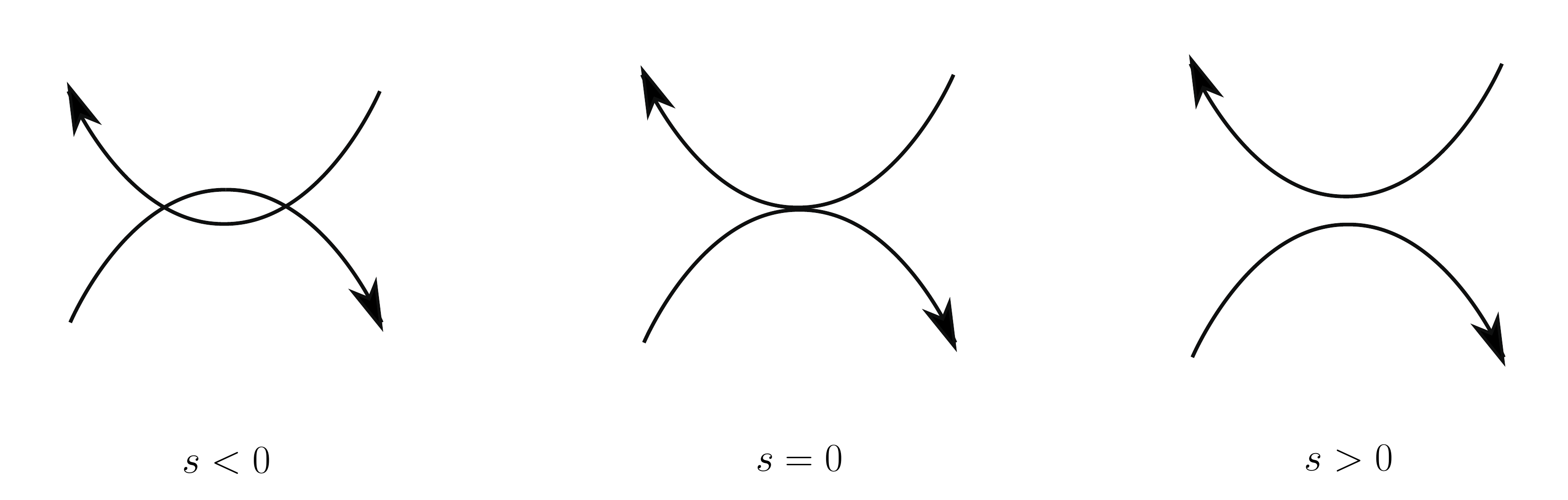}
\caption{A negative crossing through an inverse self-tangency}
  \label{inverse}
\end{subfigure}
\begin{subfigure}{0.9\textwidth}
  \centering
  \includegraphics[width=0.8\linewidth]{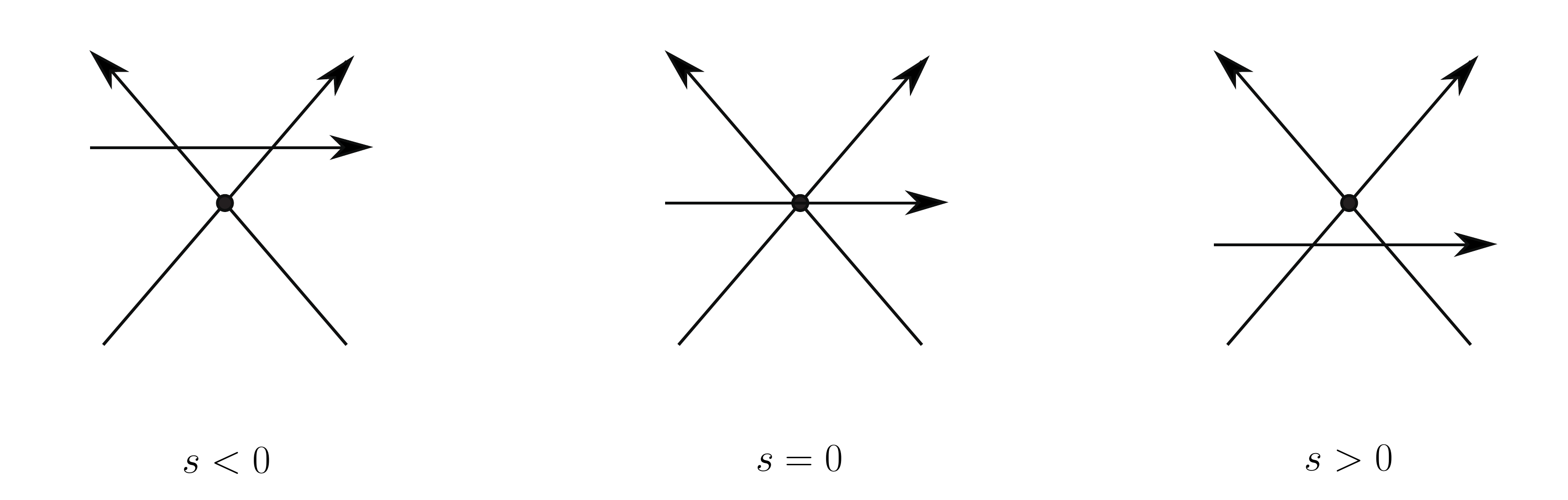}
\caption{A crossing through a triple point}
  \label{triple}
\end{subfigure}
\caption{Disasters during a Stark-Zeeman homotopy}
\label{dislihlsijstrakyeihsed}
\end{figure}

\begin{Remark}\label{rmknodirectinhismil} \rm Since periodic orbits in planar Stark-Zeeman systems are solutions of the Hamiltonian equations associated to the Hamiltonian \eqref{hamiltoniansz}, by the existence and  uniqueness theorem of O.D.E. no direct self-tangencies happen along them. 
\end{Remark}

\noindent As a special case of a Stark-Zeeman homotopy with $B \equiv 0$, we define a Stark homotopy.

\begin{Definition}\label{defSZ} \rm  A 1-parameter family $(K^s)_{s\in [0,1]}$ of closed curves in $\C$ is called a \textit{Stark homotopy} if each member is a generic immersion  in $\C \setminus \left \{ (0,0)\right\}$ except for the disasters $(II^+)$ and $(III)$ at finitely many $s\in(0,1)$, where in this case the disaster $(II^+)$ means the appearance of    the distinguished periodic orbits in  Definition \ref{deforbits}. \index{homotopy!Stark}
\end{Definition}

 Recall that Arnold's $J^+$-invariant does not change under $(I_{\infty})$, $(II^+)$ and $(III)$.  However,  it in general changes under $(I_0)$. 

 We are now in  position to define  invariants which   also do not change under the disaster $(I_0)$.   Let $K$ be a generic immersion in $\C \setminus \left \{ (0,0) \right\}$. Abbreviate by $w_0(K) \in \Z$ the  winding number of $K$ around the origin. The \textit{$\mathcal{J}_1$-invariant} of $K$ is  defined to be
$$
\mathcal{J}_1(K):= J^+(K) + \frac{w_0(K)^2}{2}. \index{Cieliebak-Frauenfelder-van Koert invariant!$\mathcal{J}_1$ invariant}
$$
Recall that the Levi-Civita mapping   is defined as the cotangent lift of the complex squaring map
$$
L : \C \setminus \left \{ (0,0 ) \right \} \rightarrow \C \setminus \left \{ (0,0 ) \right \}, \; z \mapsto z^2.
$$
 Note that the preimage $L^{-1}(K)$   is also a generic immersion in $\C\setminus \left \{ (0,0) \right \}$. By definition, the restriction of the map $L$ to the preimage $L^{-1}(K)$ is a 2-1 covering. According to the parity of the winding number $w_0(K)$, the second invariant $\mathcal{J}_2$ is defined as follows: if $w_0(K)$ is even, then the preimage $L^{-1}(K)$ consists of two connected components. We then choose one component $\widetilde{K}$ and define
$$
\mathcal{J}_2(K) := J^+(\widetilde{K}). \index{Cieliebak-Frauenfelder-van Koert invariant!$\mathcal{J}_2$ invariant}
$$
This definition is well-defined since the definition of $L$ implies that the two connected components of $L^{-1}(K)$ are related by a $\pi $-rotation in $\C \setminus \left \{ (0,0) \right \}$.  Therefore,   $\mathcal{J}_2$ does not depend on the choice of components. If $w_0(K)$ is odd, then $L^{-1}(K)$ consists of a single component and we define
$$
\mathcal{J}_2(K):= J^+( L^{-1}(K)).
$$

\begin{Proposition} { \rm (\cite[Propositions 4 and 5]{invariant})} \label{propsduiub44}  $\mathcal{J}_1$ and $\mathcal{J}_2$ are invariant under Stark-Zeeman homotopies.
\end{Proposition}

\noindent If $w_0(K)$ is even, $\mathcal{J}_1$ and $\mathcal{J}_2$ are in general completely independent, see \cite[Proposition 7]{invariant}. However, if $w_0(K)$ is odd, then they  have the following relation:

\begin{Proposition} \label{theoremsdih4} { \rm (\cite[Proposition 6]{invariant})} If $w_0(K)$ is odd, then $\mathcal{J}_2(K) = 2\mathcal{J}_1(K)-1$.
\end{Proposition}

 Our next task is to show that $\mathcal{J}_1$ and $\mathcal{J}_2$ are also invariants for Stark homotopies. By definition of a Stark homotopy, it suffices to prove that the two quantities do not change before and after the  appearance of the three distinguished periodic orbits in  Definition \ref{deforbits}. 

Let $K$ be one of the three distinguished orbits. We perturb $K$ slightly so that the perturbed orbit $\widetilde{K}$  is a generic immersion as follows: near a braking point one can take any perturbation. However for a collision, since a Stark system is close to the Kepler problem near the origin,  we perturb $K$ in such a way that the perturbed orbit encircles the origin, see Figure \ref{ledasdsammaddd2dddddsdsdsdsdsdsddssdsdd4ddd}.

 If $K$ is either a brake-brake orbit or a brake-collision orbit, then any two perturbed curves only differ from each other by orientation and the number of exterior loops and crossings through a triple point, see Figure \ref{lemmaddasdadssdsdsdsdsdsdsdsdd2dddddd4}.  By definition,  these two perturbations have the same $J^+$. 
 Thus, we can define the $J^+$-invariant $K$  by the one of $\widetilde{K}$, i.e., 
\begin{equation} \label{eq:JplusSTARK}
J^+(K):=J^+(\widetilde{K}). 
\end{equation}
\begin{figure}[h]
 \centering
 \includegraphics[width=0.45\textwidth, clip]{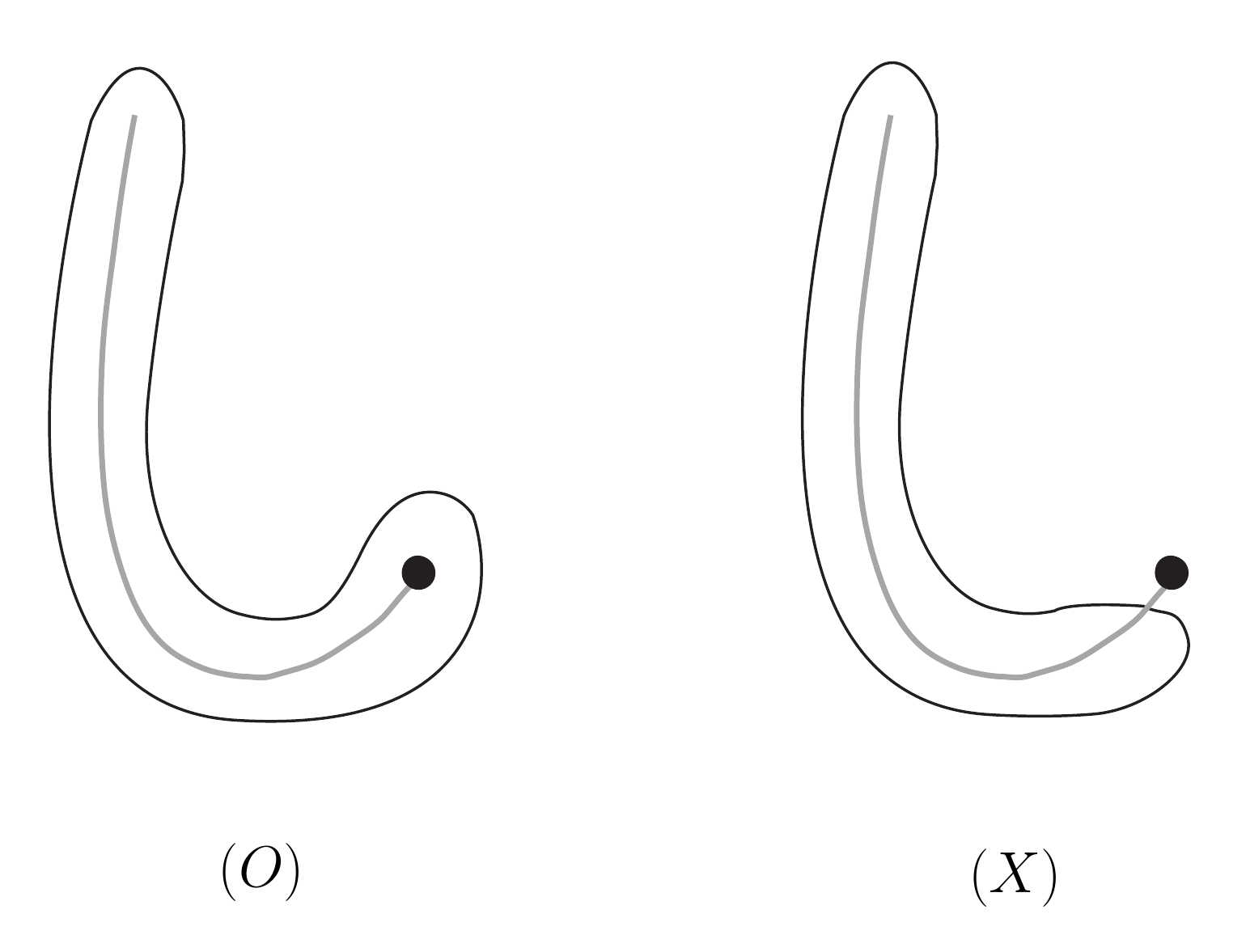}
 \caption{Two perturbations of the brake-collision orbit (gray). Since a Stark system is close to the Kepler problem when the particle moves near the origin, the particle moves as in the left figure. Thus, the $J^+$-invariant of this brake-collision orbit equals $0$.}
 \label{ledasdsammaddd2dddddsdsdsdsdsdsddssdsdd4ddd}
\end{figure}  
\begin{figure}[h]
 \centering
 \includegraphics[width=1.0\textwidth, clip]{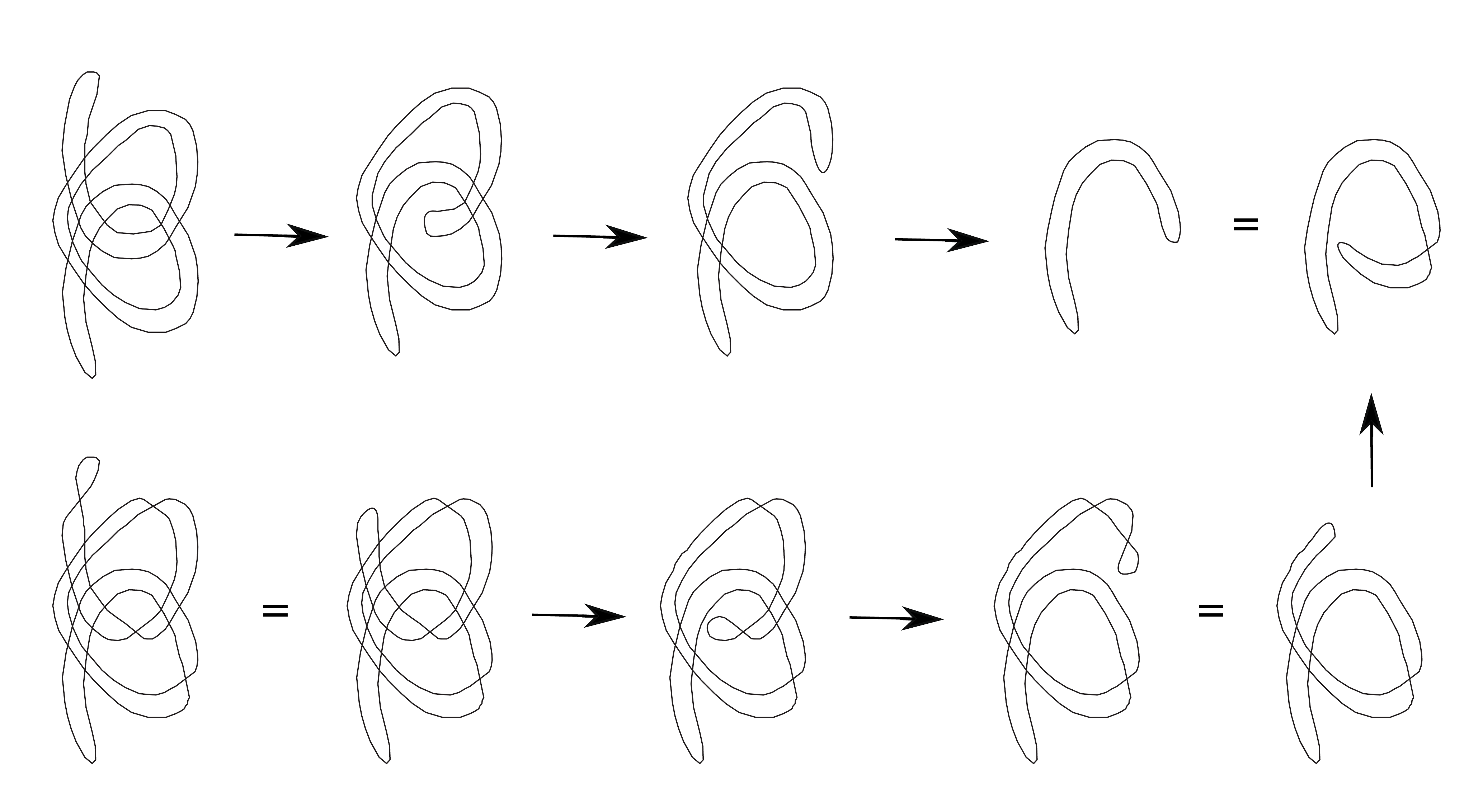}
 \caption{Two perturbations of the brake-brake orbit $K$  in Figure \ref{lemddmaddd2dddddd4ddd}. The above perturbation homotopes $K$ to the standard curve $K_1$ with four negative crossings through  a direct self-tangency and the below one with four negative crossings through a direct self-tangency,  two crossings through a triple point, and four deaths of an exterior loop. Consequently, they have the same invariant and hence the $J^+$-invariant of the brake-brake orbit is given by 8. }
 \label{lemmaddasdadssdsdsdsdsdsdsdsdd2dddddd4}
\end{figure}  
\noindent
This together with the fact that $w_0(K)^2$ does not change by small perturbations as well   imply  that   
the definitions $\mathcal{J}_i (K):=\mathcal{J}_i (\widetilde{K})$, $i=1,2$, are also well-defined.

Let $K$ be a collision-collision orbit. 
  Note   that  we have $w_0(K)^2 =0$ or $4$, depending on perturbations.   
  Likewise, the $J^+$-invariant might depend on perturbations as well
  and hence in this case \eqref{eq:JplusSTARK} is not well-defined.
However, the $\mathcal{J}_1$ invariant  is independent of  perturbations. For example, look at  Figure   \ref{ddorbit} which illustrates two possible perturbations of  $K$. 
       \begin{figure}[h]
\begin{center}
\begin{tikzpicture}[scale=1.0]
\draw [gray] (-3,1) to [out=0, in=120] (-2,0.05);
\draw[gray] (-4,0) to [out=90, in=180] (-3,1);
\draw[gray] (-3,-1) to [out=180, in=270] (-4,0);
\draw[gray] (-3,-1) to [out=0, in=210] (-2,0.05);
\draw[->,thick]  (-1.8,0) to [out=90, in=0] (-3,1.2);
\draw[thick]  (-4.2,0) to [out=90, in=180] (-3,1.2);
\draw[thick]  (-4.2,0) to [out=270, in=180] (-3,-1.2);
\draw[<-,thick]  (-1.8,0) to [out=270, in=0] (-2,-0.2);
\draw[thick]  (-2,-0.2) to [out=180, in=300] (-2.2,0.1);
\draw[thick]  (-2.2,0.1) to [out=120, in=300] (-2.4,0.35);
\draw[thick]  (-2.4,0.35) to [out=120, in=0] (-3,0.8);
\draw[thick]  (-3.8, 0) to [out=90, in=180] (-3,0.8);
\draw[<-,thick]  (-3.8, 0) to [out=-90, in=180] (-3,-0.8);
\draw[thick]  (-3 , -0.8) to [out= 0, in=210] (-2.2,0.1);
\draw[->,thick]  (-3 , -1.2) to [out= 0, in=210] (-1.88,0);
\draw[thick]  (-2.2 , 0.1) to [out= 30, in=180] (-2,0.2);
\draw[<-,thick]  (-2  , 0.2) to [out=  0, in=90] (-1.87 ,0.1);
\draw[thick]  (-1.87  , 0.1) to [out=270  , in=80] (-1.88,0 );
 \draw [fill] (-2.02,0.05) circle [radius=0.05];
\draw [gray] (3,1) to [out=0, in=120] (4,0.05);
\draw[gray] (2,0) to [out=90, in=180] (3,1);
\draw[gray] (3,-1) to [out=180, in=270] (2,0);
\draw[gray] (3,-1) to [out=0, in=210] (4,0.05);
\draw[ ->,thick]  (4.2,0) to [out=90, in=0] ( 3,0.8);
\draw[thick]  (1.8,0) to [out=90, in=180] ( 3,1.2);
\draw[thick]  (1.8,0) to [out=270, in=180] ( 3,-1.2);
\draw[<-,thick]  (4.2,0) to [out=270, in=0] (4,-0.2);
\draw[thick]  (4,-0.2) to [out=180, in=300] (3.8,0.1);
\draw[thick]  (3.8,0.1) to [out=120, in=300] (3.6,0.35);
\draw[<-,thick]  (3.6,0.35) to [out=120, in=0] ( 3,1.2);
\draw[thick]  (2.2, 0) to [out=90, in=180] ( 3,0.8);
\draw[thick]  (2.2, 0) to [out=-90, in=180] ( 3,-0.8);
\draw[thick]  ( 3 , -0.8) to [out= 0, in=210] (3.8,0.1);
\draw[<-,thick]  ( 3 , -1.2) to [out= 0, in=210] (4.12,0);
\draw[  thick]  (3.8 , 0.1) to [out= 30, in=180] (4,0.2);
\draw[thick]  (4  , 0.2) to [out=  0, in=90] (4.13 ,0.1);
\draw[thick]  (4.13  , 0.1) to [out=270  , in=80] (4.12,0 );
 \draw [fill] (3.98,0.05) circle [radius=0.05];
\draw [fill] (4,0.05) circle [radius=0.05];
 \end{tikzpicture}
\end{center}
 \caption{Two perturbations of  a collision-collision orbit. We have $w_0(K)^2=4$ for the one in the left-hand side and $w_0(K)^2 =0$ for the right one.}
 \label{ddorbit}
\end{figure}
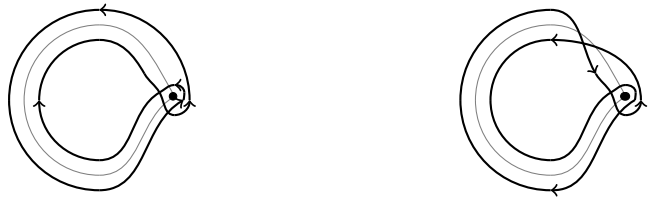
  For the left perturbation,  we have   $w_0(K)^2 =4$. During a Stark-Zeeman homotopy from $K$ to a circle, we encounter only an inverse self-tangency which implies that the perturbation has $J^+ =0$. For the right one, we have $w_0(K)^2 =0$. In this case, we  have a direct self-tangency during a Stark-Zeeman homotopy and hence the perturbation has $J^+=2$. In any case, we have $\mathcal{J}_1=2.$ The assertion for $\mathcal{J}_2$ can be shown in a similar manner.
 Thus, we also define $\mathcal{J}_i$   of $K$ as  $\mathcal{J}_i (K):=\mathcal{J}_i (\widetilde{K})$, $i=1,2$.

 \begin{Proposition} The two quantities $\mathcal{J}_1$ and $\mathcal{J}_2$ are invariants for Stark homotopies. 
\end{Proposition}
\begin{proof} Let $(K^s)_{s \in [0,1]}$ be a Stark homotopy. Suppose that $K^{s_0}$ is one of the distinguished orbits for some $s_0 \in (0,1)$. Since $K^{s_0 \pm \epsilon}$ are generic immersions and small perturbations of $K^{s_0}$, provided that $\epsilon >0$ is small enough,  we have $\mathcal{J}_i(K^{s_0- \epsilon}) =\mathcal{J}_i(K^{s_0+ \epsilon})$, $i=1,2$. 
This completes the proof of the proposition.
\end{proof}

We conclude this section with the following proposition which provides formulas for the $\mathcal{J}_1$ invariants of brake-brake orbits and brake-collision orbits and $\mathcal{J}_2$ invariants for brake-collision orbits which will be used to prove the main result   in Section \ref{secmain}. The $\mathcal{J}_2$ invariant of brake-brake orbits in the Euler problem will be given later.

 \begin{Proposition} \label{sdfkluliugliu3gsgd}  
  Let $\gamma$ be a brake-brake orbit or brake-collision orbit.   Assume that all intersection points of $\gamma$ are quadruple and the number of quadruple points equals $N$.  Then we have
$$ 
\mathcal{J}_1(\gamma)  = \begin{cases}  2N    & \text{ if $\gamma$ is a brake-brake orbit  }        \\   2N + 1/2       & \text{ if  $\gamma$ is a brake-collision orbit}   \end{cases}
$$
and
$$ 
\mathcal{J}_2(\gamma) =    4N       \quad \text{ if  $\gamma$ is a brake-collision orbit.}  
$$
\end{Proposition}
\begin{proof}  
Let  $\gamma$ be as in the assertion. By the previous argument,  one can generically homotope a brake-brake orbit to a circle which does not encircle the origin  and   a brake-collision orbit can be generically homotoped to a circle which encircles the origin. Therefore, the winding number of a brake-brake orbit around origin is given by zero and  that of a brake-collision orbit equals  either plus  or minus one. Note that  after a small perturbation, each quadruple point gives rise to four double points of a generic immersion. During a homotopy,  it leads to one crossing through a direct self-tangency and one crossing through an  inverse self-tangency (possibly with a finite number of crossings through a triple point).    Since the perturbation of  $\gamma$ is generically homotoped to a circle, we conclude that  $J^+(\gamma)=2N$. This proves the assertion for the $\mathcal{J}_1$ invariants.

  If $\gamma$ is a brake-collision orbit, then the preimage consists of a single  orbit consisting of inverse self-tangencies. By definition of the map $L$, the number of quadruple points along $L^{-1}(\gamma)$ is given by $2N$. This proves the formulas for the $\mathcal{J}_2$ invariants and completes the proof of the proposition.
\end{proof}

\begin{Remark} \rm We remark that the  formula of the $\mathcal{J}_2$ invariant for odd winding number given in the previous Proposition matches up with the relation  given in Proposition \ref{theoremsdih4}.
\end{Remark}

\section{Special cases of the planar circular restricted three body problem}\label{secorbits}

The describing Hamiltonian  of the PCR3BP $H_{3BP}  : T^*( \R^2 \setminus \left \{E, M\right \} ) \rightarrow \R$ is given by
\begin{align*}
H_{3BP} (q,p)&= \frac{1}{2}|p|^2 - \frac{1-\mu}{|q-E|} - \frac{\mu}{|q-M|} +  q_1 p_2 - q_2 p_1  \\
&=  \frac{1}{2}( (p_1 - q_2)^2 + (p_2 + q_1)^2)- \frac{1-\mu}{|q-E|} - \frac{\mu}{|q-M|}- \frac{1}{2} |q|^2 ,
\end{align*}
where $E=(-\mu,0)$,  $M=(1-\mu,0)$ and $\mu \in (0,1)$.  The point $E$ is referred to as the \textit{Earth} and the point $M$ is referred to as the \textit{Moon}. The moving particle is called the \textit{satellite}. The parameter $\mu$ is the mass ratio of the two primaries. Note that the PCR3BP is a planar Stark-Zeeman system with $A=(A_1, A_2)=(q_2,- q_1)$. In this section, we study  two friends of the PCR3BP. The first friend, which is called the \textit{rotating Kepler problem}, is obtained by switching off the Moon, i.e., we take $\mu=0$ in the PCR3BP. The second friend, which is called the \textit{Euler problem of two fixed centers}, is obtained by switching off the rotating terms $q_1p_2 - q_2p_1$.

\subsection{The     rotating Kepler problem}\label{secrkpds} In this subsection, we recall some results on the rotating Kepler problem. For more details, we refer to \cite{RKP, Kim}.

The Hamiltonian of the rotating Kepler problem is given by
\begin{equation*}
H_{\text{RKP}}(q,p) = \frac{1}{2}|p|^2 - \frac{1}{|q|} + q_1p_2 - q_2p_1.
\end{equation*}
The rotating Kepler problem is also a planar Stark-Zeeman system with $A= (q_2,- q_1)$. The Hamiltonian  $H_{\text{RKP}}$ has a unique critical value $-3/2$ which  satisfies the same properties as $c_1$ in Section \ref{secperSZ}. Indeed, for $c<-3/2$ the  Hill's region consists of two connected components: a bounded component whose closure is homeomorphic to a closed unit disk and a unbounded component.

One can easily see that the angular momentum $L=q_1p_2 - q_2 p_1$ and the (inertial) Kepler energy $E=\tfrac{1}{2}|p|^2 - \tfrac{1}{|q|}$ are  integrals of the system. In the following we assume that the Kepler energy is negative: $E<0$ and hence any (inertial) Kepler orbit  is either an elliptic orbit(including a collision orbit) or a circular orbit.   Since $L$ and $E$ Poisson commute, i.e., $\left \{ L,E \right \} =0$, where $\left \{ \cdot, \cdot \right\}$ is the Poisson bracket of smooth functions, the flows of the Hamiltonian vector fields $X_L$ and $X_E$ commute. Hence,  the Hamiltonian flow of  $H_{\text{RKP}} =L+E$ coincides with  the composition of the two Hamiltonian flows
\begin{equation} \label{eqflowcommute}
\phi_{H_{\text{RKP}}}^t = \phi_{L+E}^t = \phi_L^t \circ \phi_E^t.
\end{equation}
It follows that any orbit $\gamma^{\text{RKP}}$ in the rotating Kepler problem has the form $\gamma^{\text{RKP}}(t) = \exp(it) \gamma(t)$, where $\gamma$ is a Kepler orbit. Note that even though $\gamma$ is always periodic, $\gamma^{\text{RKP}}$ is not  necessarily periodic. 

We   assume that $\gamma$ is an ellipse of period $T>0$. For $\gamma^{\text{RKP}}$ to be periodic, a suitable resonance condition has to be satisfied:  $\gamma^{\text{RKP}}$ is periodic if and only if the periods of $\exp(it)$ and $\gamma(t)$ are commensurable, i.e., there exist positive integers $k,l\in\N$ which are relatively prime and satisfy  $ kT  = 2 \pi l$.  If this resonance condition is satisfied, the $2\pi l$-peridic orbit $\gamma^{\text{RKP}}$ is a $k$-fold covered Kepler ellipse in an $l$-fold covered coordinate system.  This observation gives rise to the following definition. 
\begin{Definition} \rm (\cite[Section 4]{RKP})   A $\lambda$-periodic orbit $\gamma^{\text{RKP}}(t) = \exp( it ) \gamma(t)$   is called a \textit{$T_{k,l}$-type orbit} if $\lambda = 2\pi l$ and $\gamma$ is a  Kepler ellipse of period $T$ satisfying $kT = 2 \pi l$  for some relatively prime $k,l \in \N$. A Liouville torus on which $T_{k,l}$-type orbits lie is called a \textit{$T_{k,l}$-torus}. 
Finally,  a smooth family of $T_{k,l}$-tori  is called the \textit{$T_{k,l}$-torus family}. 
\end{Definition}

\begin{Remark} \rm Whenever we consider a $T_{k,l}$-torus, we assume that $k$ and $l$ are relatively prime which means that the $T_{k,l}$-type orbits are simple covered.
\end{Remark}

\noindent Again by the fact that $L$ and $E$ Poisson commute, along the $T_{k,l}$-torus family the Kepler energy $E$ is constant. Indeed, the Kepler energy $E_{k,l}$ of the $T_{k,l}$-torus family is given by  
$$ E_{k,l}=-\frac{1}{2}\bigg( \frac{k}{l} \bigg)^{\frac{2}{3}},$$
see for example \cite[Section 6]{RKP} or \cite[Lemma 3.3]{Kim}. Throughout this paper, we make the following assumption

\; 

\textbf{Assumption.}  \textit{$E_{k,l}<-1/2$, or equivalently  $k>l$.  }

\;

\noindent It turns out  that the $T_{k,l}$-torus family bifurcates from the $(k-l)$-fold covered circular orbit of angular momentum $L=-(l/k)^{1/3}$ and dies at $(k+l)$-fold covered circular orbit of angular momentum $L=(l/k)^{1/3}$, see for example \cite[Section 6 and Appendix B]{RKP} or \cite[Proposition 3.4]{Kim}. We call the circular orbits of positive momentum and negative angular momentum   the \textit{retrograde circular orbit} and the \textit{direct circular orbit}, respectively.

\subsection{The   Euler problem}
The describing Hamiltonian of the Euler problem  is given by
\begin{equation}\label{hameuler}
H_{\text{Euler}}(q,p) = \frac{1}{2} |p|^2  -\frac{1-\mu}{|q-E|}- \frac{\mu}{|q-M|},
\end{equation}
where $E=(0,0)$ and $M=(1,0)$. It  has a unique critical value $c=-1-2\sqrt{\mu(1-\mu)}$ which also satisfies the same properties as $c_1$ in Section \ref{secperSZ}.  In the following we continue to denote the critical value by $c_1$.  Note that for $c<c_1$ the Hill's region consists of two bounded components: one is around the Earth, which is abbreviated by $\mathcal{K}_c^{\text{E}}$, and the other is around the Moon, which is abbreviated by $\mathcal{K}_c^{\text{M}}$. The Euler problem  is also  integrable and an integral is given by
\begin{equation}\label{integraleuler}
B(q,p) = -(q_1 p_2 - q_2 p_1)^2  + (q_1 p_2 - q_2 p_1)p_2 - \frac{ (1-\mu)q_1}{|q-E|} - \frac{ \mu ( 1-q_1)}{|q-M|}.
\end{equation}

We apply the translation $(q_1, q_2, p_1,p_2) \mapsto (q_1 - {1}/{2}, q_2 , p_1 , p_2)$ under which the dynamics does not change. Note that we now have $E=(-{1}/{2},0)$ and $M=({1}/{2},0)$. We introduce the double covered elliptic coordinates $(\lambda, \nu) \in \R \times S^1[-\pi, \pi]$ which are defined by the relations
\begin{eqnarray*}
\cosh \lambda=  |q- \text{E} | + |q- \text{M}|  \;\;\;\text{ and  }\;\;\; \cos \nu=   |q- \text{E}| - |q- \text{M}|   .
\end{eqnarray*}
The momenta ${p_{\lambda} }$ and ${ p_{\nu}}$ are determined by the canonical relation $ p_1 d q_1 + p_2 dq_2 = p_{\lambda} d \lambda + p_{\nu} d\nu$. Note that 
\begin{equation}\label{eqtrnasofrmation}
(\lambda, \nu) \mapsto (q_1, q_2) = \bigg(\frac{1}{2} \cosh \lambda \cos \nu,   \frac{1}{2} \sinh \lambda \sin \nu\bigg)
\end{equation}
is  a 2-to-1 branched covering with branch points at $E, M$. The two sheets are related by $(\lambda , \nu) \mapsto (- \lambda, - \nu)$. This involution extends to the phase space by 
\begin{equation}\label{phaseinvolusion}
(\lambda, \nu,p_{\lambda}, p_{\nu}) \mapsto (-\lambda, -\nu, -p_{\lambda}, -p_{\nu}).
\end{equation}

The Hamiltonian in the elliptic coordinates is given by
\begin{equation*}
H_{\text{Euler}} = \frac{H_{\lambda} + H_{\nu}}{\cosh^2 \lambda - \cos^2 \nu},
\end{equation*}
where $H_{\lambda} = 2p_{\lambda}^2 - 2 \cosh \lambda$ and $H_{\nu} = 2p_{\nu}^2 + 2(1-2\mu)\cos \nu$. Following the convention by Strand-Reinhardt \cite{Strand}, we choose the first integral by $G = -H_{\text{Euler}} +2B$
\begin{equation*}
G  = - \frac{ H_{\lambda} \cos^2 \nu + H_{\nu} \cosh^2 \lambda}{\cosh^2 \lambda - \cos^2 \nu }.
\end{equation*}
In this paper we consider negative energy values $H_{\text{Euler}}=c<0$ so that every motion is bounded. The classically allowed region in the lower half $(G,H_{\text{Euler}} )=(g,c)$-plane is given in Figure \ref{fig:region}. Points in the four regions, labeled by $P$, $L$, $S$, and $S'$, are regular values of the map $(G,H_{\text{Euler}}) : T^*\R^2 \rightarrow \R^2$  and points on the five black curves, which are given by
\begin{eqnarray*} 
&&\ell_{1,2} : c=-g\pm2(1-2\mu),\;\;\; \;\;\; \;\;\; \;\;\; \;\;\; \ell_3 : c=-g-2,\\
&& \ell_4 : gc=(1-2\mu)^2,~c_1<c<-(1-2\mu),\;\;\; \; \ell_5 : gc=1,~-1<c,
\end{eqnarray*} 
are its critical values. 
\begin{figure}[h]
 \centering
 \includegraphics[width=0.6\textwidth, clip]{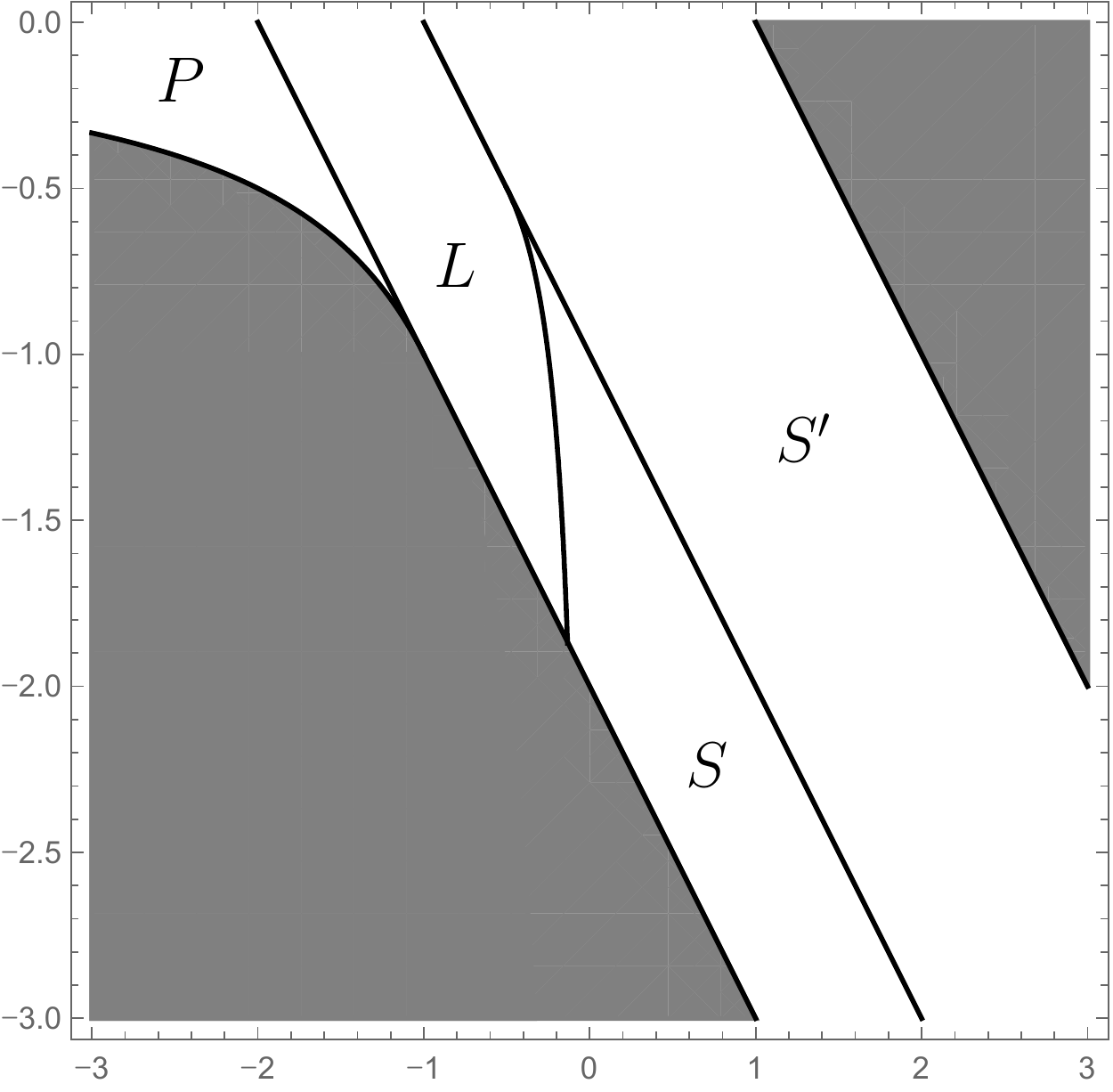}
 \caption{ For negative energies,   each point  in the four regions $S'$, $S$, $L$ and $P$ represents regular motions along which $dH_{\text{Euler}}$ and $dG$ are linearly independent. Along the black curves, the two differentials are linearly dependent.  The shaded regions are classically forbidden. For more details on the motions, see \cite{Kim2, Bifurcation}.}
 \label{fig:region}
\end{figure} 
Note that $c_1$, $c=-(1-2\mu)$ and $c=-1$ are the enegy values at which $\ell_3$ and $\ell_4$, $\ell_3$ and $\ell_5$, and $\ell_2$ and $\ell_4$, respectively.  This shows that for $c<c_1$ only the two regions $S$ and $S'$ appear. For more details, we refer to  \cite{homoclinic, Strand, Bifurcation}.

Given $(G, H_{\text{Euler}}) = (g,c)$, the momenta $p_{\lambda}$ and $p_{\nu}$ are expressed by
\begin{equation}\label{eqmomentadouble}
p_{\lambda}^2 = \frac{ c \cosh^2 \lambda + 2 \cosh \lambda +g }{2} \quad \text{ and } \quad p_{\nu}^2 = \frac{ -c \cos^2 \nu - 2(1-2\mu)\cos\nu - g}{2}.
\end{equation}
Fix $c<c_1$ and consider points  in the  region $S$ or $S'$. The phase portrait for $\lambda$   is then a simple closed curve which is symmetric with respect to both $\lambda$- and $p_{\lambda}$-axes and centered at $(\lambda, p_{\lambda}) = (0,0)$. If $(g,c)\in S$, the $\nu$-phase portrait  consists of  a disjoint union of two  simple closed curves: one is associated to the Earth component and the other is associated to the Moon component. The curve corresponding to the Earth component (or the Moon component) is symmetric with respect to the $\nu$-axis and the line $p_{\nu}=-\pi$ (or the line $p_{\nu}=0$) and centered at $(\nu, p_{\nu}) = (0,-\pi)$ (or $(\nu, p_{\nu}) = (0,0)$). If $(g,c)\in S'$, only one simple closed curve corresponding to the Earth component appears, see Figure \ref{phaseportrati}. 
\begin{figure}[h]
\begin{subfigure}{0.45\textwidth}
  \centering
  \includegraphics[width=0.9\linewidth]{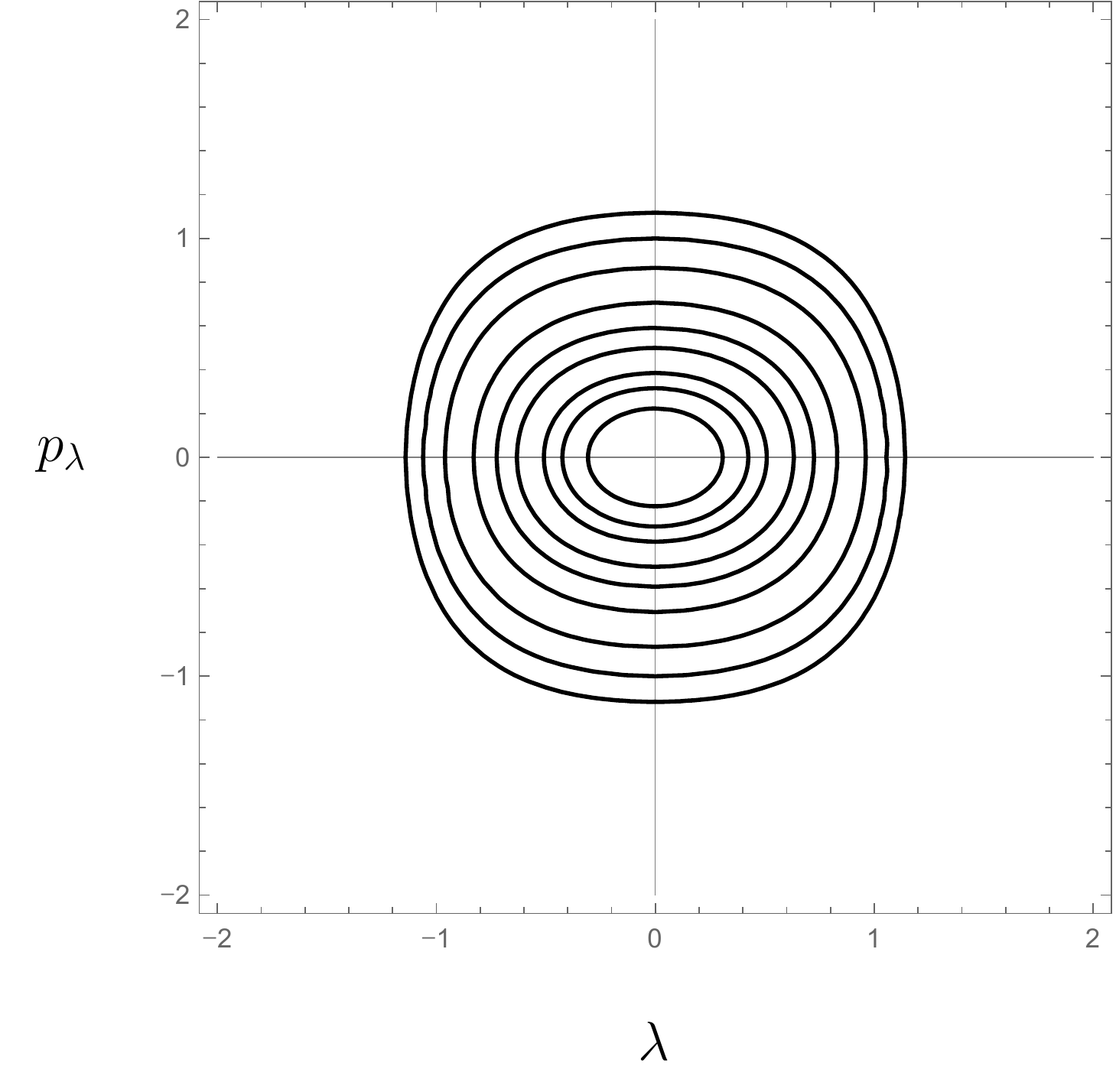}
\caption{Phase portraits for $\lambda$}
  \label{pahsel}
\end{subfigure}
\begin{subfigure}{0.45\textwidth}
  \centering
  \includegraphics[width=0.9\linewidth]{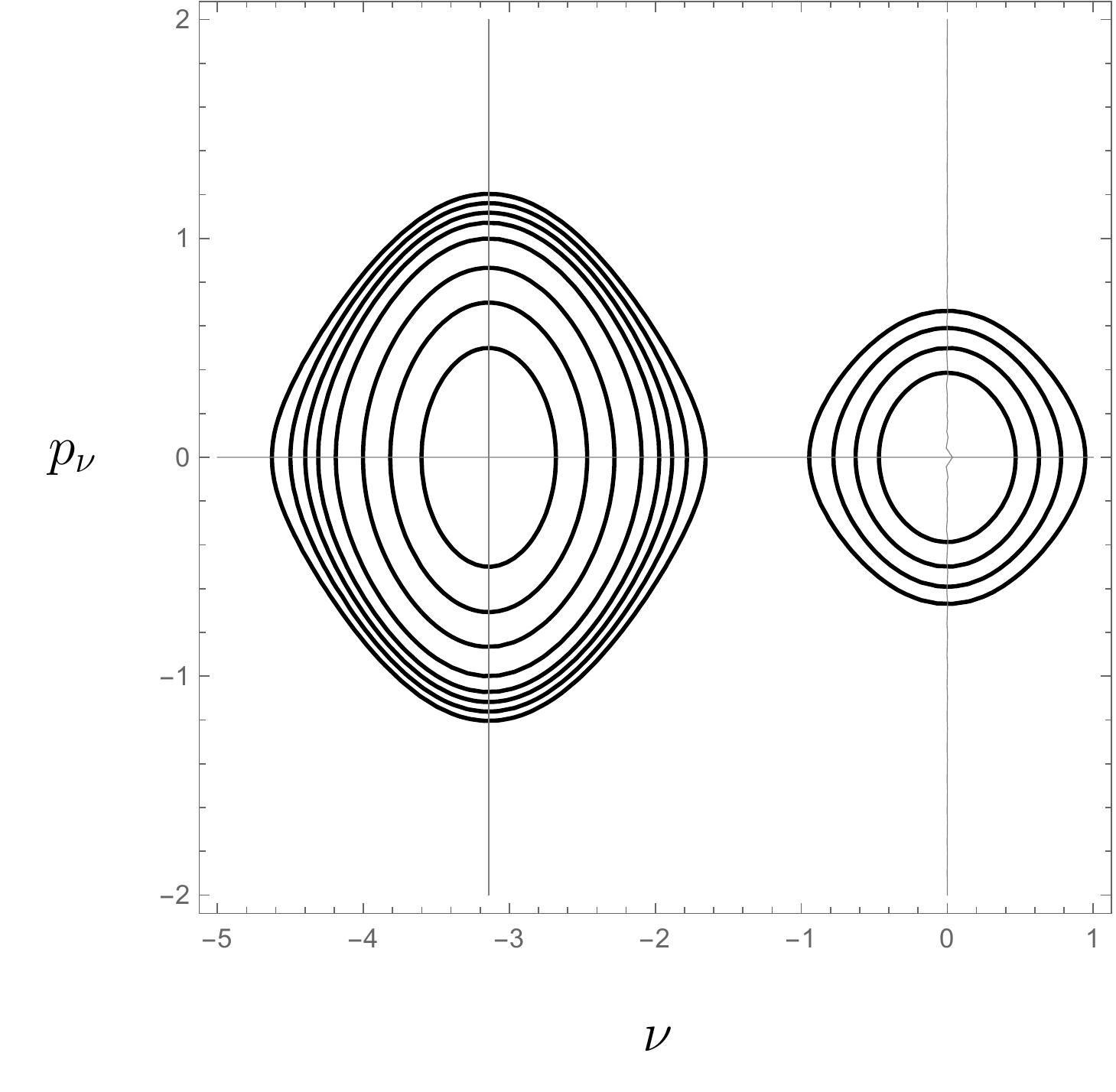}
\caption{Phase portraits for $\nu$}
  \label{pasheo}
\end{subfigure}
\caption{Phase portraits for $\mu = \tfrac{1}{4}$}
\label{phaseportrati}
\end{figure}
This implies that the preimage of a  regular level $(g,c) \in S$ or $(g,c) \in S'$ represents    a disjoint union of two tori or a single torus, respectively.  Therefore, in the $S$-region the satellite moves near either the Earth or the Moon while in the $S'$-region it is confined to a neighborhood of the Earth.  In view of \eqref{eqmomentadouble}, the  $\lambda$-phase portraits are oriented in counterclockwise for $\left \{ \lambda >0 \right \}$ and in clockwise for $\left \{ \lambda <0 \right \}$. The  $\nu$-phase portraits are oriented  in a similar way. This picture holds  because the two sheets of the double covering of the phase space are related with each other by \eqref{phaseinvolusion}.

Fix an energy level $H_{\text{Euler}} = c<0$. Due to collisions, the energy hypersurface $H^{-1}_{\text{Euler}}(c)$ is noncompact. However, one can regularize the dynamics on the energy hypersurface: define the new Hamiltonian
$$
K := (H_{\text{Euler}} - c) (\cosh^2 \lambda - \cos^2 \nu)= K_{\lambda} + K_{\nu},
$$
where
$K_{\lambda}= 2p_{\lambda}^2 - 2 \cosh \lambda - c \cosh^2 \lambda$ and $K_{\nu}= 2p_{\nu}^2 + 2 (1-2\mu)\cos \nu + c \cos^2 \nu$. Then  orbits of $H_{\text{Euler}}$ with energy $c$ and the time parameter $t$ correspond to orbits of $K$ with energy $0$ and the time parameter $\tau$
$$
d\tau = \frac{dt}{\cosh^2 \lambda - \cos^2}.
$$
In particular, the satellite is now allowed to pass through the primaries.

Since $K_{\lambda}$ and $K_{\nu}$ Possion commute, as in the rotating Kepler problem we have
$$
\phi_{H_{\text{Euler}}}^t = \phi_{K_{\lambda}}^t \circ \phi_{K_{\nu}}^t.
$$
It follows that for an orbit to be periodic, we need a resonance condition for the $\lambda$-period $T_{\lambda}$ and the $\nu$-periods $T_{\nu}$. More precisely, an orbit is periodic if and only if the ratio $R=T_{\nu}/T_{\lambda}$, which is called the \textit{rotation number}, is rational.  It turns out that the rotation number only depends on the value $(G,H_{\text{Euler}})=(g,c)$, i.e., every periodic on a given Liouville torus has the same rotation number. Varying $(g,c)$ we obtain the \textit{rotation function} $R$. The rotation functions of Liouville tori are computed by means of complete elliptic integrals of the first kind, see \cite{Dullin, Kim2}.   We do not study  the rotation function $R$ in detail but  give the following definitions: fix $k$ and $l$ which are relatively prime.  A Liouville torus with $R = {k}/{l}$  is called a \textit{$T_{k,l}$-torus}. Periodic orbits which lie on a $T_{k,l}$-torus is called \textit{$T_{k,l}$-type orbits}.  Fixing $R={k}/{l}$ and varying $(G,H_{\text{Euler}})=(g,c)$ gives rise to a smooth family of $T_{k,l}$-tori, which will be referred to as the \textit{$T_{k,l}$-torus family}, cf. Section \ref{secrkpds}.  We illustrate some periodic orbits in Figure  \ref{d}.

\begin{figure}[h]
\begin{subfigure}{0.32\textwidth}
  \centering
  \includegraphics[width=1.0\linewidth]{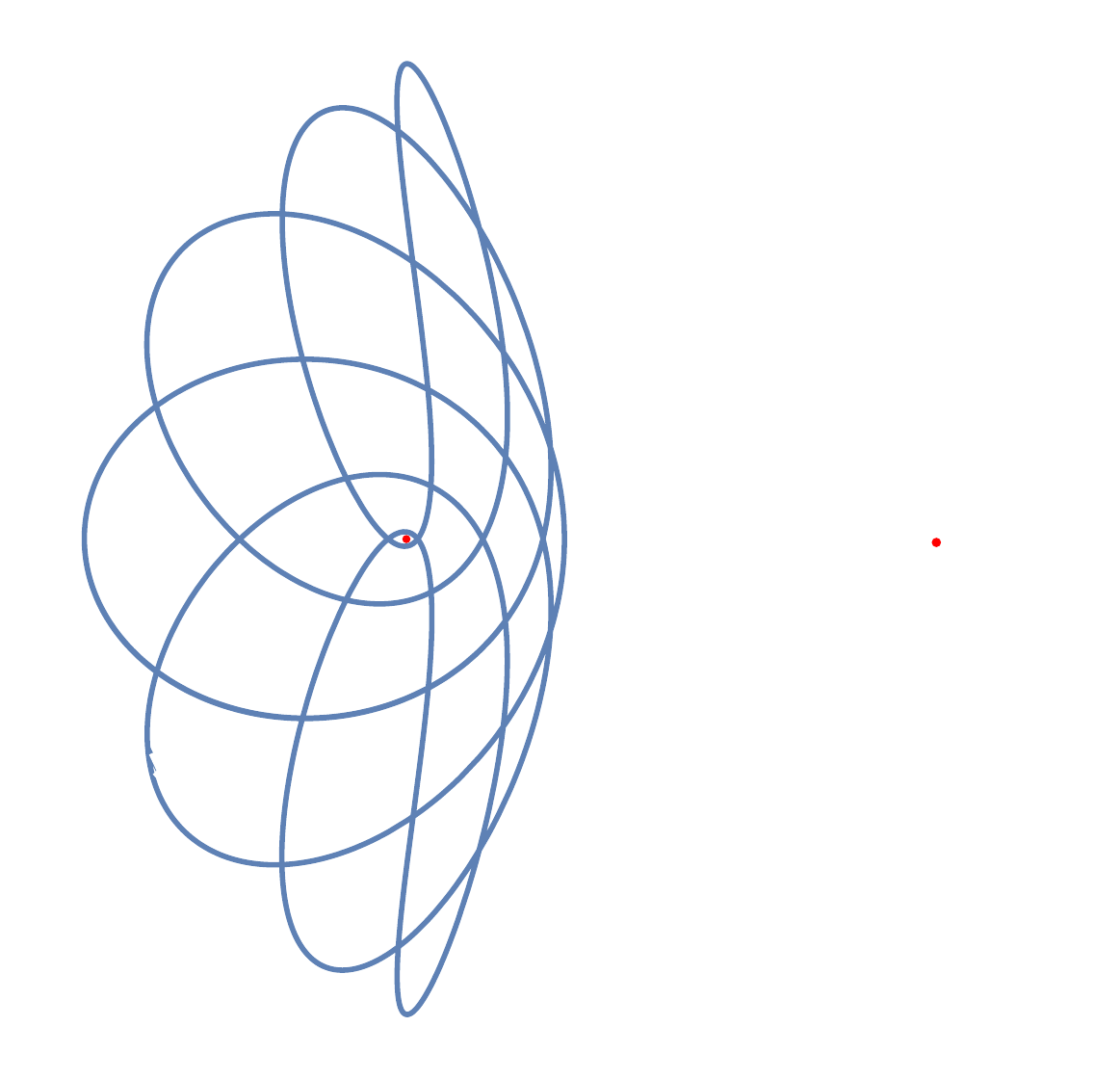}
\caption{$R=\tfrac{7}{5}$}
\end{subfigure}
\begin{subfigure}{0.32\textwidth}
  \centering
  \includegraphics[width=1.0\linewidth]{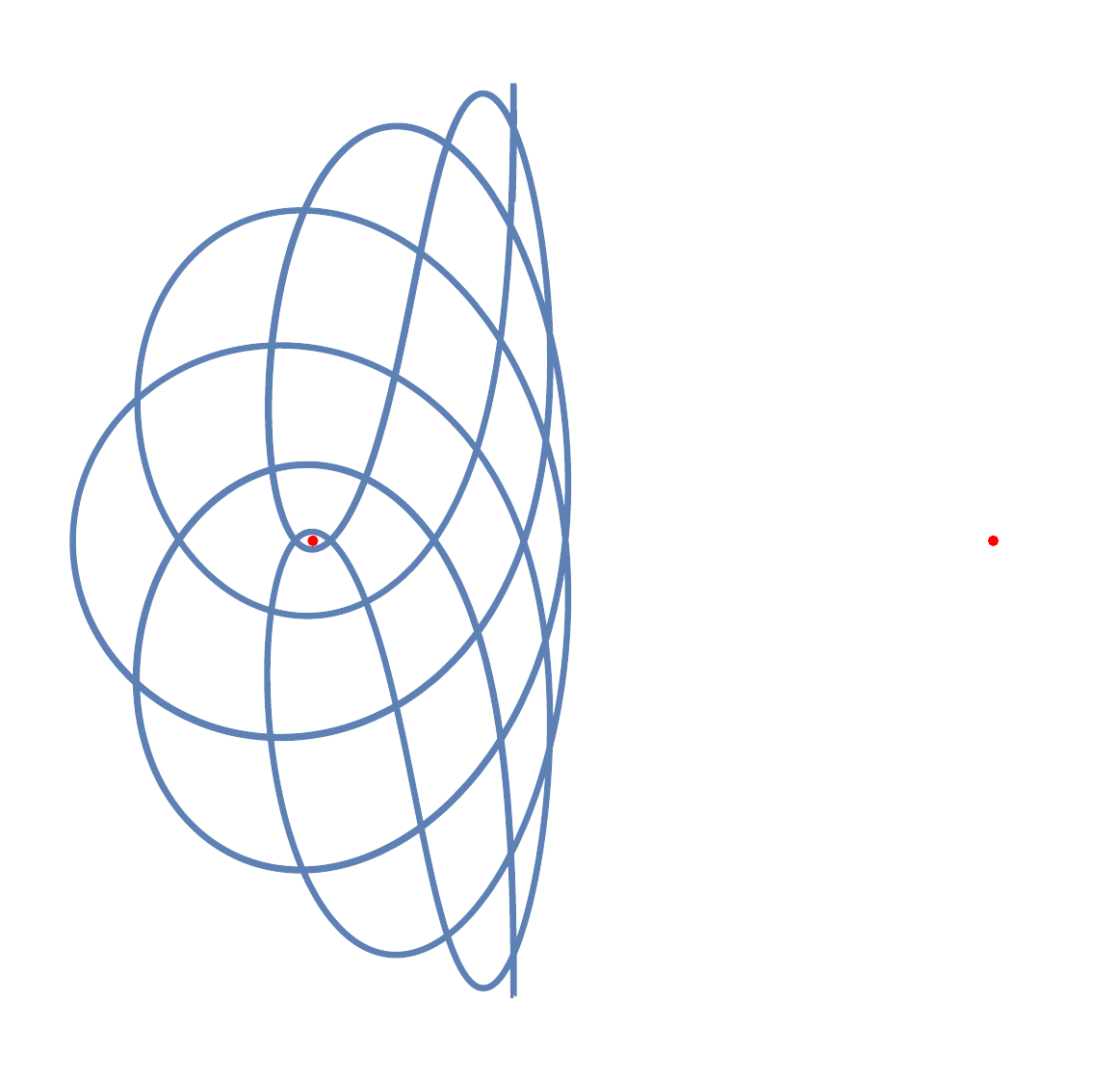}
\caption{$R=\tfrac{8}{5}$}
\end{subfigure}
\begin{subfigure}{0.32\textwidth}
  \centering
  \includegraphics[width=1.0\linewidth]{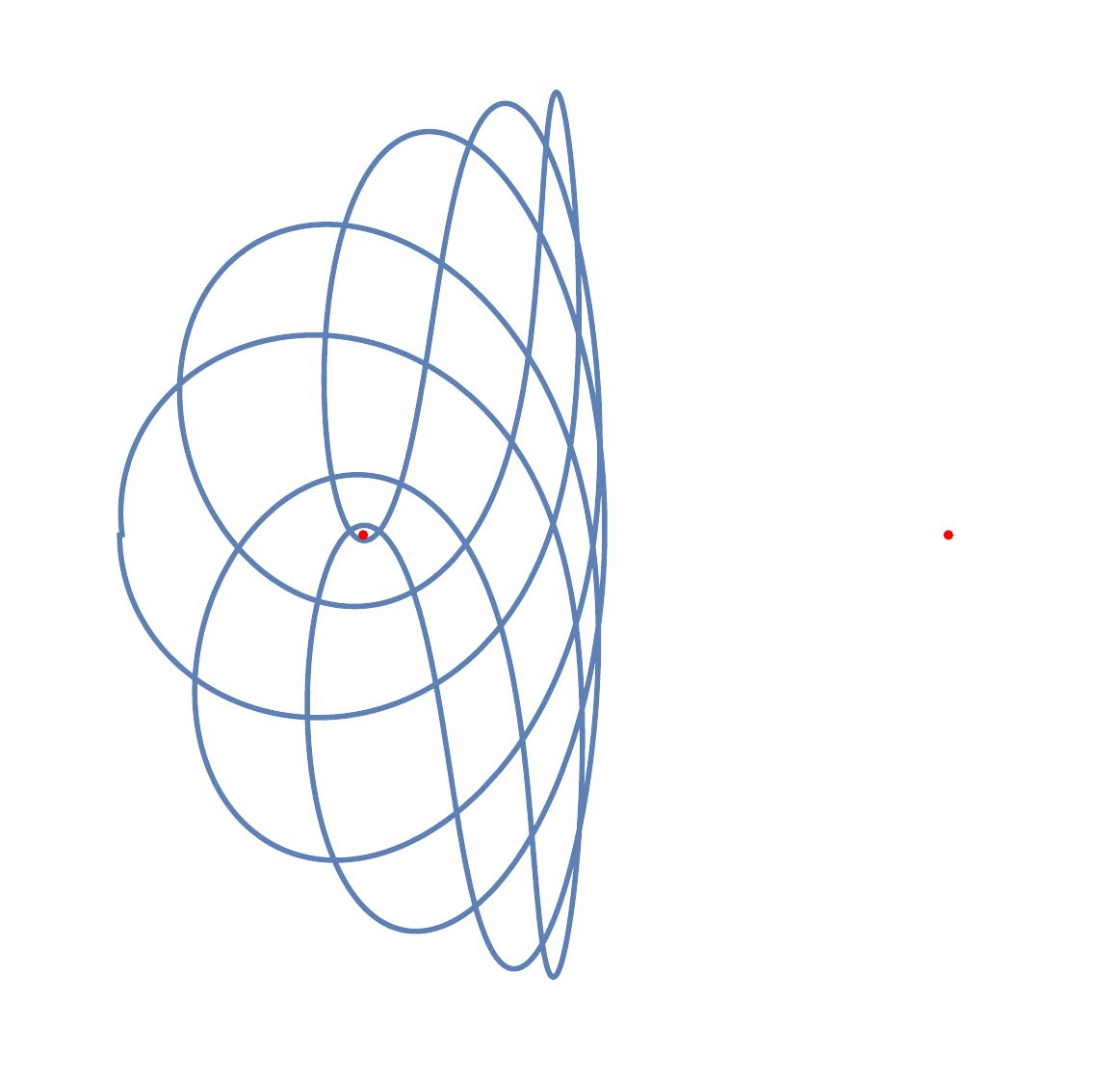}
\caption{$R=\tfrac{9}{5}$}
\end{subfigure}
\caption{Some symmetric periodic orbits in the Euler problem}
\label{d}
\end{figure}

In the following we only consider  energy levels $c<c_1$.  Note that the energy hypersurface $H_{\text{Euler}}^{-1}(c)$  consists of two bounded components: one is around the Earth and the other is around the Moon. By abuse of notation, we also call them the Earth component and the Moon component, respectively.   On each  component, there exist precisely two critical periodic orbits along which $dG$ and $dH_{\text{Euler}}$ are linearly dependent: the exterior and interior collision orbits, see Figure \ref{coll}. 
\begin{figure}[h]
 \centering
 \includegraphics[width=0.6\textwidth, clip]{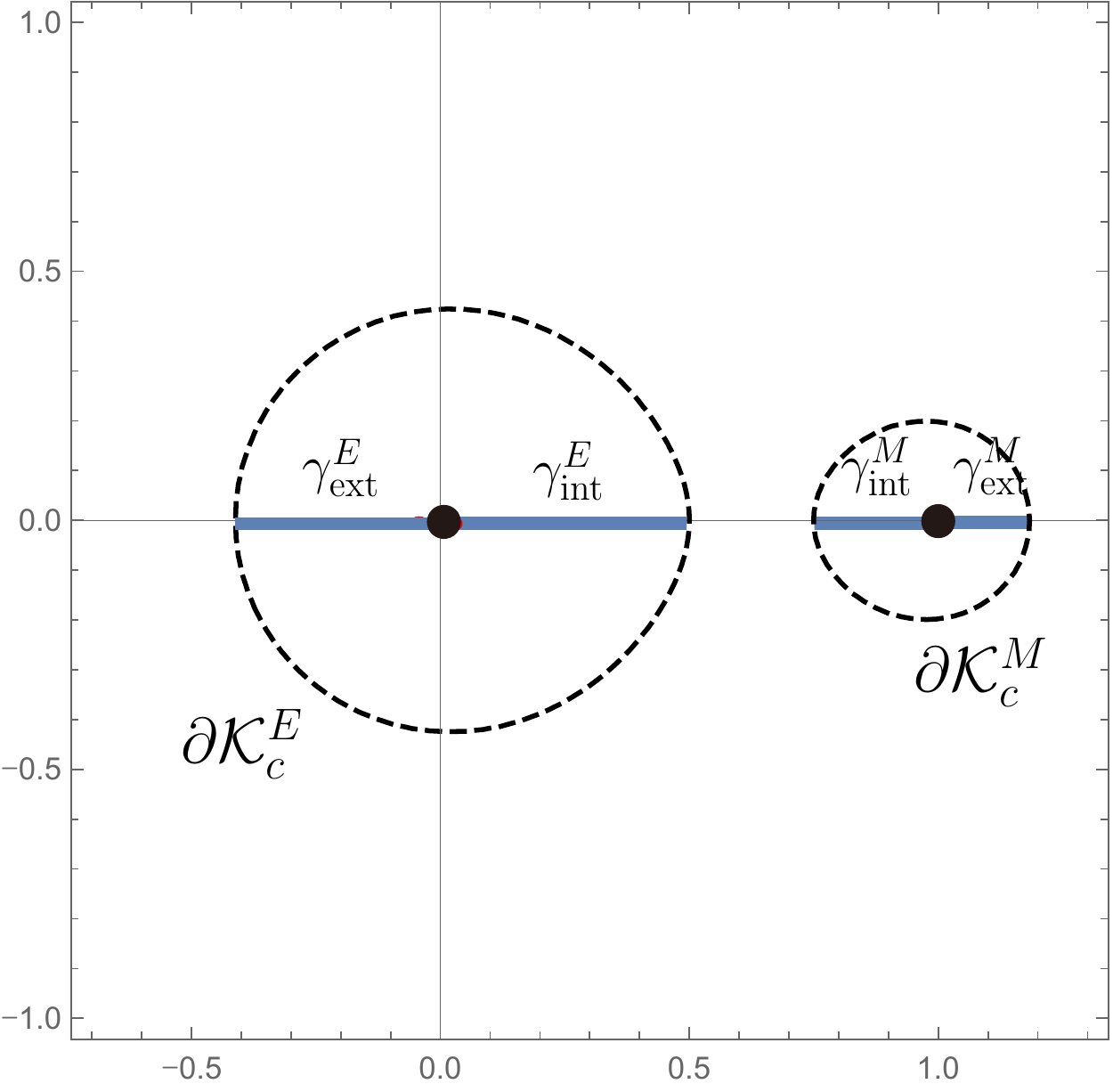}
 \caption{Exterior and interior collision orbits}
 \label{coll}
\end{figure}

\begin{Remark} \rm In the Euler problem, the exterior collision orbit and  interior collision orbit play roles as the retrograde circular orbit  and  direct circular orbit in the rotating Kepler problem, respectively, see \cite{Kim2}.
\end{Remark}

\begin{Remark}\label{remarkproperties} \rm The properties of families of periodic orbits in planar Stark systems described in Section \ref{szstyemsd}  can be obtained by means of the elliptic coordinates. Fix $(G,H_{\text{Euler}} )=(g,c)$. Without loss of generality we may consider only motions near the Earth.  The assertions will not depend on $\mu$.

\;

\noindent \textit{For collisions}: suppose that the satellite collides with the Earth, i.e., we have  $(\lambda, \nu) =(0, - \pi)$. Plugging this point into \eqref{eqmomentadouble} gives rise to 
\begin{equation}\label{juhee}
 p_{\lambda}=  \pm \sqrt{ \frac{  c+g+2}{2} } \quad \text{ and } \quad  p_{\nu} =  \pm \sqrt{\frac{  -c-g+2(1-2\mu)}{ 2}}.
\end{equation}
Since $-2 < g+c < -2(1-2\mu)$ for $(g,c) \in S$ and $-2(1-2\mu) < g+c < 2(1-2\mu)$ for $(g,c) \in S'$,  we see from \eqref{juhee} that a collision orbit exists on any $T_{k,l}$-torus  in the $S$- or the $S'$-region. Moreover, we observe that  $p_{\lambda}$ and  $p_{\nu}$ vanish only if  $(g,c) \in \ell_3$ and $(g,c) \in \ell_1$, respectively. Thus, both momenta never vanish at collisions along torus-type orbits. Instead,  $\lambda$ and $\nu$  change  signs before and after collisions.  The reflection symmetries of the phase portraits then imply that   the satellite retraces its former journey after the collision.

\;

\noindent \textit{For touching the boundary of the Hill's region}: in a similar way as above, we see that every $T_{k,l}$-torus corresponding to $(g,c) \in S$ or $(g,c) \in S'$ contains a periodic orbit which admits the condition $(p_{\lambda}, p_{\nu}) =(0,0)$ at which the satellite  hits the   boundary of the Earth component.   Moreover,  the momenta $p_{\lambda}$ and $p_{\nu}$ change  signs before and  after $(p_{\lambda}, p_{\nu}) =(0,0)$.  Again by the symmetries of the phase portraits    the satellite retraces its former journey after touching the boundary $\partial \mathcal{K}_c^E$,  cf. Lemma \ref{lemmabakc1}.

\;

\noindent \textit{For inverse self-tangencies}: suppose that a periodic  orbit has an inverse self-tangency. Since $(\lambda, \nu) \mapsto (q_1, q_2)$ is a branched double covering whose two sheets are related by $(\lambda , \nu) \mapsto (-\lambda, -\nu)$ and the phase portraits are symmetric to the position axes, i.e., $p\mapsto -p$, it follows that the inverse self-tangency is not isolated. Then by the compactness of the image of the periodic orbit  we conclude that every point on the orbit under consideration is an inverse self-tangency,  cf. Lemma \ref{lemmainverse}.

\;

\noindent \textit{For intersection points}:  consider  an intersection point along a torus-type orbit $\gamma$. Then by \eqref{eqmomentadouble} it is at most a quadruple point. Assume that it is a triple point: there exist $t_0, t_1, t_2 \in S^1$ such that $\gamma(t_0)=\gamma(t_1)=\gamma(t_2)$. Again by \eqref{eqmomentadouble} we may assume that $\dot{\gamma}(t_0) = - \dot{\gamma}(t_1)$ and $\dot{\gamma}(t_2) \neq \pm  \dot{\gamma}(t_0)$. It follows that $\gamma$ is   either a brake-brake, a collision-collision or a brake-collision orbit and hence there must exist $t_3\in S^1$ such that $\dot{\gamma}(t_3) \neq \pm  \dot{\gamma}(t_0)$ and $\dot{\gamma}(t_3) = -  \dot{\gamma}(t_2)$. This contradicts the assumption and hence the intersection point must be either a double point or a quadruple point. Moreover, if $\gamma$ is not one of the three distinguished orbits, then it admits only double points.
 \end{Remark}

Following the exposition given by  Verhaar \cite[Section 5]{thesis} we prove   
\begin{Proposition} \label{lemmacollision} Assume that $k>l$ are relatively prime and fix any $T_{k,l}$-torus associated to $(g,c)\in S$ or $(g,c)\in S'.$
\begin{enumerate}[label=(\roman*)]
\item  it   contains precisely two collision orbits  which can be  obtained from each other by the $q_1$-axis reflection.  If $k+l$ is even, they are brake-collision orbits and  if $k+l$ is odd, they are collision-collision orbits;
\item if $k+l$ is odd, then it  contains a unique brake-brake orbit which is symmetric with respect to the $q_1$-axis. If $k+l$ is even, there exist no brake-brake orbits.
\end{enumerate}
\end{Proposition}
\begin{proof} $(i)$  Note that   $R = {T_{\nu}}/{ T_{\lambda}} =  {k}/{l}$ implies that   the satellite has $k$ cycles in $\lambda$ and $l$ cycles in $\nu$. Abbreviate  $T= k T_{\lambda}=l T_{\nu} $. Suppose that $\gamma(t) = (\lambda(t), \nu(t))$ admits a collision.   We choose the initial condition  to be the    collision:   $\gamma(0)  =(0, -\pi)$. Without loss of generality, we may assume that  $(p_{\lambda}(0), p_{\nu}(0)) = (p_{\lambda}^{\max}, p_{\nu}^{\max})$, where $p_{\lambda}^{\max}, p_{\nu}^{\max} >0$.  Assume that the second collision happens  at $t={T}/{2}$ from which we obtain that  $(p_{\lambda}(T/2) , p_{\nu}(T/2) ) = (p_{\lambda}^{\max}, p_{\nu}^{\max})$, $(p_{\lambda}^{\max}, -p_{\nu}^{\max})$, $(-p_{\lambda}^{\max}, p_{\nu}^{\max})$ or $(-p_{\lambda}^{\max}, -p_{\nu}^{\max})$, see Figure \ref{lemma35}. 
\begin{figure}[h]
 \centering
 \includegraphics[width=0.8\textwidth, clip]{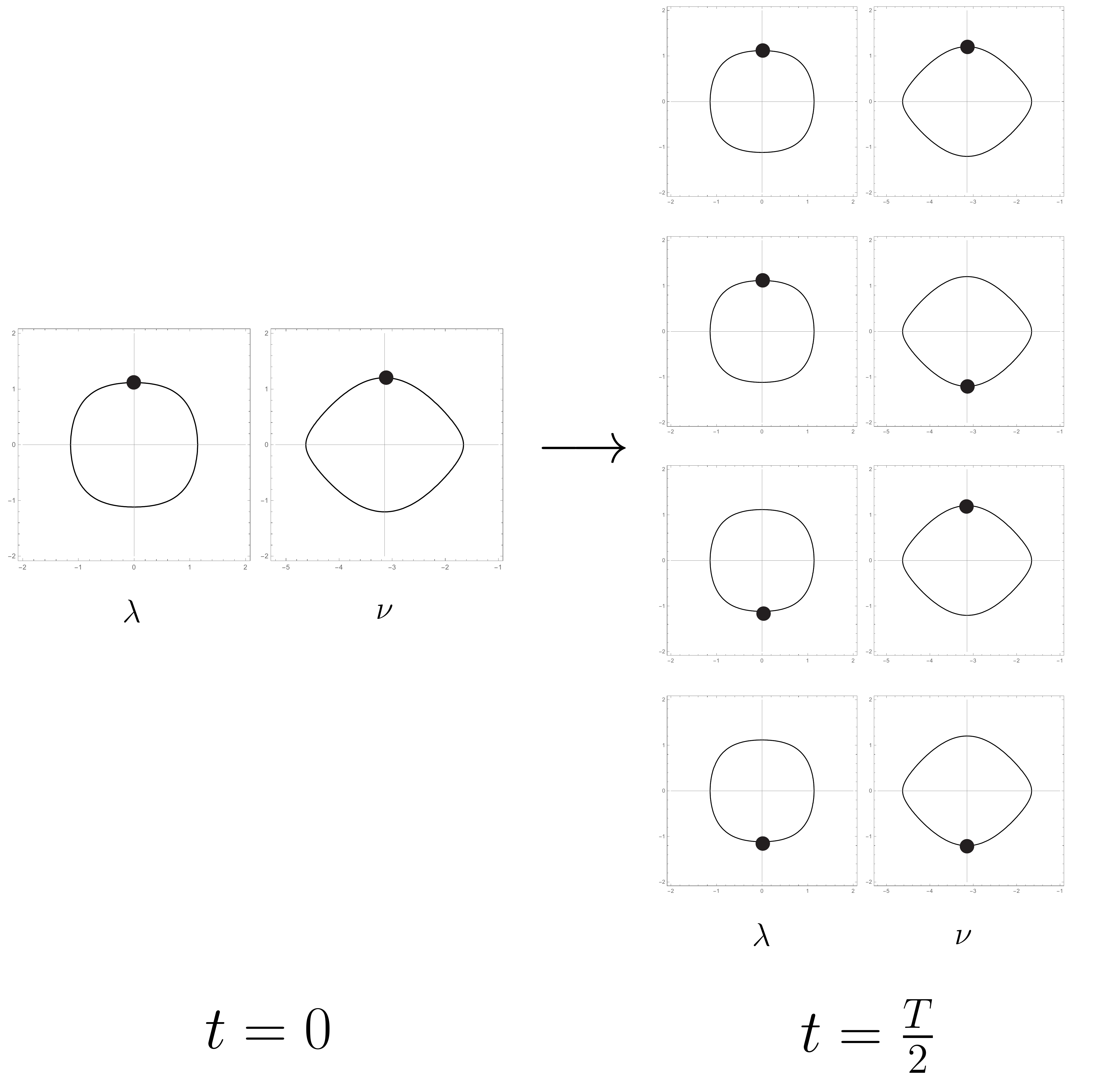}
 \caption{Four possibilities for the collision at $t={T}/{2}$. The first case is excluded since $\text{gcd}(k,l)=1$. The second and the third cases imply that the orbit under consideration is a (simple covered) collision-collision orbit. The last case means that the orbit is a (double covered) brake-collision orbit.}
 \label{lemma35}
\end{figure}
Note that for each $\sigma = \lambda, \nu$ that  $p_{\sigma} \mapsto p_{\sigma}$ and $p_{\sigma} \mapsto -p_{\sigma}$ at $t={T}/{2}$ imply that $ {T}/{2}$ is an   even  multiple and an  odd multiple of ${T_{\sigma}}/{2}$, respectively. It follows immediately that the first case   $(p_{\lambda}(T/2) , p_{\nu}(T/2) ) = (p_{\lambda}^{\max}, p_{\nu}^{\max})$ contradicts the fact that $k$ and $l$ are relatively prime.

Assume that the second case $(p_{\lambda}(T/2) , p_{\nu}(T/2) ) = (p_{\lambda}^{\max}, -p_{\nu}^{\max})$ which implies that $k$ is even and $l$ is odd.  In this case $\gamma$ is a simple covered $T$-periodic collision-collision orbit. Observe  that the Hamiltonian $K$ admits the following  anti-symplectic involutions:
$$
I_1 : (\lambda, \nu, p_{\lambda}, p_{\nu} ) \mapsto (\lambda, \nu, -p_{\lambda}, -p_{\nu} )
$$
and 
$$
I_2 : (\lambda, \nu, p_{\lambda}, p_{\nu} ) \mapsto (-\lambda, \nu,  p_{\lambda}, -p_{\nu} ).
$$
The first involution correspond to the time reversal under which the image of $\gamma$ does not change. In view of \eqref{eqmomentadouble} the second involution corresponds to the $q_1$-axis reflection. The equations \eqref{juhee} then show that there exist precisely two collision-collision orbits which are obtained from each other by the $q_1$-axis reflection.  For the third case, we obtain the same result. 

For the remaining case $(p_{\lambda}(T/2) , p_{\nu}(T/2) ) = (-p_{\lambda}^{\max}, -p_{\nu}^{\max})$ both $k$ and $l$ are odd. In this case the satellite comes back to the collision at $t=T/2$ by retracing its former journey. Since $t=T/2$ is the time at which the second collision happens, this shows that there exists a unique $t_0 \in (0,T/2)$ at which the satellite touches the boundary of the Hill's region. Therefore, as a $T$-periodic orbit $\gamma$ is a  double covered  brake-collision orbit.  By the same reasoning  as in the previous case, there exist precisely two brake-collision orbits and one is obtained from the other by the reflection with respect to the $q_1$-axis. This first assertion is proved.

$(ii)$ Assume that $\gamma$ is a $T$-periodic orbit which has  a braking point   at $t=0$, i.e., $(p_{\lambda}(0), p_{\nu}(0))=(0,0)$. Without loss of generality, we may assume that $(\lambda(0), \nu(0))=( \lambda^{\max}, \nu^{\max})$, $  \lambda^{\max}, \nu^{\max}>0$.   As in the proof of the first assertion,  we have  four possibilities for the second braking at $t=T/2$: $(\lambda(T/2), \nu(T/2))=( \lambda^{\max}, \nu^{\max})$, $( \lambda^{\max}, -\nu^{\max})$, $( -\lambda^{\max}, \nu^{\max})$, or $( -\lambda^{\max}, -\nu^{\max})$. In a similar way, we see that the first case is impossible and the last case gives rise to a doubly-covered $T$-periodic brake-collision orbit. From the second and third cases we obtain a unique simple covered $T$-periodic brake-brake orbit which is symmetric with respect to the $q_1$-axis.  This proves the second assertion and completes the proof of the proposition.
\end{proof}
\begin{Remark} \rm   That any $T_{k,l}$-torus in the $S$-region contains precisely two collision orbits was already proved by Dullin and Montgomery by means of symbolic dynamics, see \cite[Corollary 8]{Dullin}. They also observed the existence of brake-brake orbits, see \cite[Section 9]{Dullin}.
\end{Remark}

We conclude this section with the following lemma.

\begin{Lemma}\label{lemmacoversing} Fix any $T_{k,l}$-torus family. If $k+l$ is even, then it   bifurcates from the $l$-fold  covered interior collision orbit and dies at the $k$-fold covered exterior collision orbit. If $k+l$ is odd, then  it  bifurcates from the $2l$-fold covered interior collision orbit and ends at $2k$-fold covered exterior collision orbit. 
\end{Lemma}
\begin{proof} It suffices to determine the intersection numbers of  $T_{k,l}$-type orbits with   the negative and positive $q_1$-axis. By \cite[Theorem 1]{Dullin} these numbers only  depend on $k$ and $l$. Therefore, we just need to choose suitable representatives. In view of the previous proposition, we choose a brake-collision orbit if $k+l$ is even and a collision-collision orbit if $k+l$ is odd.  Recall that along a $T_{k,l}$-type orbit $\gamma$ the variable $\lambda$ makes $k$ cycles and the variable $\nu$ makes $l$ cycles. Since $(\lambda, \nu)$  is a 2-1 (branched) covering, this shows that $\gamma$ intersects the positive $q_1$-axis  precisely $2k$-times  at which we have  $\lambda=0$ and the negative $q_1$-axis precisely $2l$-times at which we have $\nu=-\pi$. The proof of Proposition \ref{lemmacollision} shows that   we see that $\gamma$ is double covered and single covered if $k+l$ is even and if $k+l$ is odd, respectively.   This finishes the proof of the lemma.
\end{proof}

\section{Main argument}\label{secmain}

In this section we shall prove the main result of this paper. 

\subsection{Invariants for the rotating Kepler problem}  In view of the equation  \eqref{eqflowcommute}, we see that on a fixed $T_{k,l}$-torus, the trajectories of any two orbits can be obtained from each other by a rotation in the $q$-plane. In particular, they have the same shape. This implies that on each $T_{k,l}$-torus the invariants  $\mathcal{J}_1$ and $\mathcal{J}_2$ are  constant. In other words, in studying these invariants without loss of generality we may consider each $T_{k,l}$-torus family as a one-parameter family.

Recall that we assume that  $E=E_{k,l}<1/2$ (or equivalently $k>l$).  

\begin{Theorem}\label{maintheoremRKP} {\rm (\cite[Theorem 1.1 and Proposition 3.7]{Kim})} Fix relatively prime $k>l$.  The $T_{k,l}$-torus family in the rotating Kepler problem is a Stark-Zeeman homotopy and  its invariants are given by
$$
\mathcal{J}_1 (T_{k,l})=   1 - k + (k^2-l^2)/2
$$
and
$$
\mathcal{J}_2 (T_{k,l})= \begin{cases} (k-1)^2-l^2 & \text{ if $k+l$ is odd} \\ 1 - k + (k^2-l^2)/4 & \text{ if $k+l$ is even.} \end{cases}
$$
\end{Theorem}

\begin{Remark} \rm As mentioned in \cite[Remark 3.15]{Kim}, one can prove the previous theorem in an elementary way (but with more work) by means of the ideas of Ptolemy and Copernicus. The elementary proof will be given in \cite{Kimphd}.
\end{Remark}

\;\;\;

\subsection{Invariants for the Euler problem} From now on we will compute the invariants of the $T_{k,l}$-torus families in the Euler problem. Without loss of generality, we focus on the Earth component. Note that on each $T_{k,l}$-torus, periodic orbits form a $S^1$-family. In view of  Remark \ref{remarkproperties} and  Proposition  \ref{lemmacollision}  the $S^1$-family is a Stark   homotopy.  Consequently, in view of calculating the invariants,  as for the rotating Kepler problem, we may consider each $T_{k,l}$-torus family, which is in fact a two-parameter family,  as a one-parameter family.  The same remark and lemmas imply

\begin{Proposition}\label{sdsdghomo} Each $T_{k,l}$-torus family in the $S$- or $S'$-region is a Stark-Zeeman homotopy.
\end{Proposition}
\noindent Thus, in order to calculate the invariants for the $T_{k,l}$-torus families in the Euler problem it suffices to compute for a suitable  periodic orbit in that familiy. We choose a brake-brake orbit for $k+l$ odd and a brake-collision orbit for $k+l$ even. In view of  Proposition \ref{sdfkluliugliu3gsgd} it remains to determine the number of quadruple points.

\begin{Proposition}\label{propoeeuloerodd} Fix a $T_{k,l}$-torus. The following are true:
\begin{enumerate}[label=(\roman*)]
\item if $k+l$ is odd,  then the brake-brake orbit  has precisely $ {(k-1)(l-1)}+ (k+l-1)/2$ quadruple points;
\item if $k+l$ is even, then the two brake-collision orbits have precisely $ {(k-1)(l-1)}/4$ quadruple points.
\end{enumerate}
\end{Proposition}
\begin{proof}  (i) We only consider the case that $k$ is even and $l$ is odd. The other case    can be  proved in a similar way.  Recall that all intersection points along the brake-brake orbit $\gamma$ are quadruple points. Let $T>0$ be the minimal period of $\gamma$. Note that the restriction  ${\gamma}|_{[0,T/2]}$ has the same image with ${\gamma}$ and a quadruple point of $\gamma$ is a double point of $\gamma|_{[0,T/2]}$. In the following we study the restriction $\gamma|_{[0,T/2]}$ instead of $\gamma$. By abuse of notation, we use the same symbol $\gamma$ for the restriction and  we identify the orbit $\gamma$ with its image. 

 Recall that $\gamma$ intersects the $q_1$-axis if and only if we have $\lambda=0$ or $\nu=-\pi$. Since $\lambda$ and $\nu$ make $k/2$ cycles  and $l/2$ cycles along $\gamma$, respectively, we obtain that $\lambda=0$ and $\nu=-\pi$ are attained precisely $k$ times and $l$ times, respectively. Since  $\gamma$ is a brake-brake orbit which is symmetric with respect to the $q_1$-axis, those points give rise  to $(k+l-1)/2$ double points and a single point of $\gamma$ on the $q_1$-axis.

 Let $\gamma^{\pm} = \gamma \cap \left\{   \pm q_1 \geq 0  \right\}$ be the positive and negative parts of $\gamma$, respectively,  so that $\gamma = \gamma^+ + \gamma^-$. Since $\gamma$ is symmetric with respect to the $q_1$-axis,  $\gamma^+$ and $\gamma^-$ have the same number of double points and they do not have any double points on the $q_1$-axis. Therefore,  it suffices to count double points on $\gamma^+$.   Note that the period of $\gamma^+$ is given by $T/4$.

Choose the initial point of $\gamma^+$ by the braking point $(\lambda, \nu) = ( \lambda^{\max}, \nu^{\max})$, i.e., the point at which the satellite touches the boundary   $ \partial \mathcal{K}_c^{\text{E}}$. Note that along $\gamma^+$ the variables $\lambda$ and $\nu$ make $k$ and $l$ quarter-cycles, respectively.   Each quarter-cycle for $\lambda$ (or for $\nu$) corresponds to increase or decrease of $\lambda$ (or of $\nu$) between $\lambda=0$ and $\lambda=\lambda^{\max}$ (or between $\nu=-\pi$ and $\nu=\nu^{\max}$). 

By abuse of notation and for the sake of convenience, we use the symbol $\nu$ for $\nu +\pi$ so that the collision with the Earth corresponds to $(\lambda, \nu)=(0,0)$.  We now view the variables $\lambda$ and $\nu$ as functions of time. More precisely, in the graphs the horizontal axis represents the time duration $t \in [0, T/4]$ and the vertical axis represents the values of $\lambda$ or $\nu$.    Since we are considering the positive part of $\gamma$, we  reflect the negative part of the graphs with respect to the horizontal axis, see Figure \ref{figukluilu}. Since $k$ is even and $l$ is odd,   at the rightmost points of the graphs, i.e., the points at $t=T/4$, we have $\lambda=\lambda^{\max}$ and $\nu=0$. Note that the point $\gamma(T/4)$ is the single point of $\gamma$ which lies on the $q_1$-axis.

\;\;

\textit{Claim 1. Assume that  $t=t_0 \in (0,T/4)$ represents a double point of the positive part $\gamma^+$. Then we have $t_0 \in (T/4kl)\Z$.}\\
Abbreviate $\lambda_0=\lambda(t_0)$ and $\nu_0=\nu(t_0)$. We  find 
$$
t = t_1, \frac{2lT }{4kl} \pm t_1, \frac{4lT }{4kl} \pm t_1, \cdots, \frac{ (k-2)lT }{4kl}  \pm t_1, \frac{ klT }{4kl}  - t_1
$$
at which $\lambda = \lambda_0$, where $0<t_1<T/4k$, and
$$
t = t_2, \frac{2kT}{4kl} \pm t_2, \frac{4kT }{4kl}\pm t_2, \cdots, \frac{ (l-1)kT}{4kl}\pm t_2 , 
$$
at which we have $\nu = \nu_0$, where $0< t_2 < T/4l$. We only consider the case that 
\begin{equation}\label{asdlkugiu33sdfsdf3}
t_0= \frac{ 2ilT}{4kl}+ t_1 = \frac{2jkT}{4kl} +t_2,
\end{equation}
for some $ 0 \leq i \leq k/2 -1$ and $ 0\leq j\leq (l-1)/2$ from which we obtain 
\begin{equation}\label{asdlkugiu333}
 t_1 = \frac{2(jk-il)T}{4kl} +t_2.
\end{equation}
The other cases can be proved in a similar way. Since $(\lambda_0, \nu_0)$ is a double point, there must exist $m\neq i $ and $n \neq j$ satisfying either $(i)$  $ 2mlT/{4kl} +t_1=  {2nkT}/{4kl} +t_2$, $(ii)$   $ {2mlT}/{4kl} +t_1= {2nkT}/{4kl} -t_2$, $(iii)$ $ {2mlT}/{4kl} -t_1=  {2nkT}/{4kl} +t_2$, or $(iv)$ $ {2mlT}/{4kl} -t_1=  {2nkT}/{4kl} -t_2$.

Assume the first case from which it follows that
$$
t_1=  \frac{2(nk-ml)T}{4kl} +t_2.
$$
This together with \eqref{asdlkugiu333}  give rise to
$$
  \frac{2(nk-ml)T}{4kl} +t_2= \frac{2(jk-il)T}{4kl} +t_2 \;\; \;\; \Rightarrow \;\; \;\;(n-j)k =(m-i)l.
$$
Since $k$ and $l$ are relatively prime, this implies that $k$ and $l$ divide  $m-i$ and $n-j$, respectively. However, this is not the case since $|m-i|<k$ and $|n-j|<l$. Thus, the first case is impossible. A similar result holds for the last case.  We now assume the second case. The third case can be proved in a similar way. Proceeding as  the first  case we obtain that
$$
 t_1 = \frac{ ((n+j)k -(i+m)l)T}{4kl} \;\;\; \text{ and } \;\;\;  t_2 = \frac{ ((n-j)k +(i-m)l)T}{4kl}
$$
and hence in view of \eqref{asdlkugiu33sdfsdf3} it follows that
$$
t_0 = \frac{ ((n+j)k +(i-m)l)T}{4kl}.
$$
 This proves the claim.

\;\;

We   divide the time interval $ [0,T/4]$ by $kl$ subintervals such that each subinterval has length $T/4kl$, see Figure \ref{figukluilu}. Consider the $kl+1$ points $t=jT/4kl$, $j=0,1,2, \cdots, kl$. By Claim 1, each double point of $\gamma^+$ must correspond to one of these points. The $(k+1)$ points $t=iT/4k$, $0 \leq i \leq k$, correspond to the maximum or  minimum of $\lambda$ and the $(l+1)$ points $t=iT/4l$, $0 \leq i \leq  l$, correspond to the maximum or  minimum of $\nu$. Since $k$ and $l$ are relatively prime, we have
\begin{equation}\label{u6fu6zvukz6v}
\left\{\frac{iT}{4k} :  i=0,1,2, \cdots, k \right \} \cap \left\{\frac{iT}{4l} :  i=0,1,2, \cdots, l \right \} = \left \{ 0, \frac{T}{4} \right \}.
\end{equation}
It is obvious that the   points corresponding to the maximum of   $\lambda$ or $\nu$  do not represent double points of $\gamma^+$.  For the points  which correspond to the minimum and which represent points of $\gamma^+$ on the $q_1$-axis, we already showed that except for one point, they are double points. The following claim  shows that the remaining $(k-1)(l-1)$ points correspond to  double points. Once this is proved, the first assertion of the proposition follows.

\;\;

\textit{Claim 2. Among the $(kl+1)$ points described as above, $(k-1)(l-1)$ points, which do not represent the maximum or  minimum of $\lambda$ or $\nu$, correspond to double points of $\gamma^+$.}\\
We fix $t _0= NT/4kl$ for some $ 0 < N < kl$ which is not contained in the two sets in the left-hand side of \eqref{u6fu6zvukz6v}. Abbreviate  $(\lambda_0, \nu_0) = (\lambda(t_0), \nu(t_0))$. As in the previous claim, we find 
$$
A = \left\{\frac{mT}{4kl}, \frac{2lT \pm mT}{4kl}, \frac{4lT \pm mT}{4kl}, \cdots, \frac{ (k-2)lT \pm mT}{4kl}, \frac{ klT - mT}{4kl}  \right\} 
$$
at which $\lambda = \lambda_0$, where $1 \leq m \leq l-1$, and 
$$
B = \left\{\frac{nT}{4kl}, \frac{2kT \pm nT}{4kl}, \frac{4kT \pm nT}{4kl}, \cdots, \frac{ (l-1)kT \pm nT}{4kl}\right\}
$$
at which we have $\nu = \nu_0$, where $1 \leq n \leq k-1$. We need to show that $\# A \cap B=2$. Since every intersection point of $\gamma^+$ is double, it suffices to show that $A$ and $B$ have an intersection point other than $t_0$.

As in Claim 1, we only consider the case $N=2al+m=2bk+n$ for some $0 \leq a \leq k/2 -1$ and $0 \leq b \leq (l-1)/2$ from which we obtain 
\begin{equation}\label{a7gkligsdf}
2(al-bk)=n-m.
\end{equation}
 The other cases can be proved in a similar way. We observe that there exist no $r\neq a$ and $s\neq b$ satisfying $2 r l+m= 2  s k +n$ or $2 r l -m= 2  s k -n$  since $k$ and $l$ are relatively prime. On the other hand, since  $1 \leq m \leq  l-1 $, there exist $P, Q \in \Z$ such that $m = kP - lQ$. We then define $i,j$ to be
$$
i=\begin{cases} P - b & \text{ if $P>b$} \\ b-P & \text{ if $P<b$}  \\ 0 & \text{ if $P=b$}\end{cases} \;\;\; \text{ and } \;\;\; j=\begin{cases} Q - a & \text{ if $Q>a$} \\ a-Q & \text{ if $Q<a$} \\ 0 & \text{ if $Q=a$.} \end{cases} 
$$
Consider the case  $(i,j)= (P-b, Q-a)$. We then have $m=k(i+b)-l(j+a)$ and $n=k(i-b)-l(j-a)$ from which we obtain   $2jl+m=2ik-n$. Therefore, we have
$$
\lambda_0 = \lambda\bigg( \frac{ 2al T +mT}{4kl} \bigg) = \lambda \bigg(   \frac{2jlT+mT}{4kl} \bigg)
$$
and
$$
\nu_0 =\nu\bigg( \frac{ 2al T +mT}{4kl} \bigg) = \nu \bigg(   \frac{2jlT+mT}{4kl} \bigg).
$$
It remains to show that $a \neq j$. Assume by contradiction that $a=j$. We then have $n=k(i-b)$.  Since $ 1 \leq n \leq k-1$, this is not the case. This shows that $(\lambda_0, \nu_0)$ is a double point. For the cases   $(i,j)= (P-b, 0)$, $ (b-P, a-Q)$, or $ (0, a-Q)$, the assertion can be  proved in a similar way. The other five cases never happen.   This proves the claim and hence the first assertion.

\;

(ii)  Since the two brake-collision orbits are related by the $q_1$-axis reflection, without loss of generality we may choose one of them, say $\gamma$. Since $k$ and $l$ are relatively prime, both $k$ and $l$ are odd. As in the proof of the previous case,  by abuse of notation we use the symbol $\gamma$ for the restriction $\gamma|_{[0,T/2]}$ which has the same image as $\gamma$. Different from the previous case the  points of $\gamma$  on the $q_1$-axis are not necessarily double points since $\gamma$ is not symmetric with respect to the $q_1$-axis. As before,  we consider $\lambda$ and $\nu$ as functions of time. Since we are not considering the positive part of $\gamma$, but $\gamma$ itself, we do not need to reflect the negative part of the graphs. Then a  similar argument as in the proof of the first case proves the second assertion. This completes the proof of the proposition.
\end{proof}

\begin{Remark} \rm The assertions of the previous proposition hold for all brake-brake orbits or brake-collision orbits in any separable Stark systems, provided that the phase portrait of each variable is given by a simple closed curve which is symmetric with respect to both horizontal and vertical  axes.  
\end{Remark}

\begin{Remark} \rm  The proof of the previous proposition carries over to symmetric (with respect to the $q_1$-axis) periodic orbits for the case $k+l$ is even:   the corresponding symmetric periodic orbit on a $T_{k,l}$-torus has precisely $(k-1)(l-1) + (k+l-2)/2$ double points.
\end{Remark}

\begin{Example} \rm

\begin{figure}[h]
\begin{subfigure}{0.49\textwidth}
  \centering
  \includegraphics[width=1.1\linewidth]{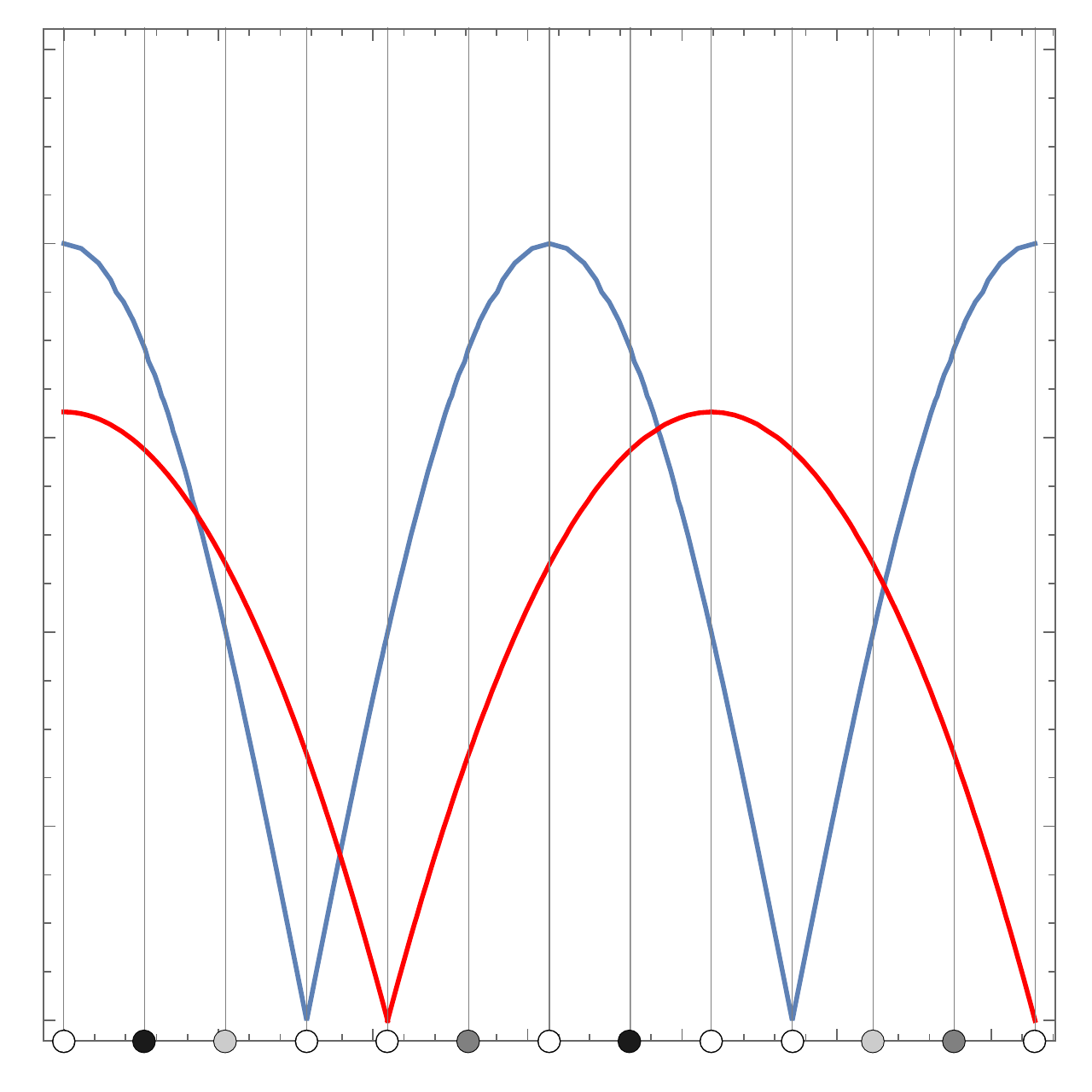}
\caption{ }
\label{figukluilu}
\end{subfigure}
\begin{subfigure}{0.49\textwidth}
  \centering
  \includegraphics[width=0.62\linewidth]{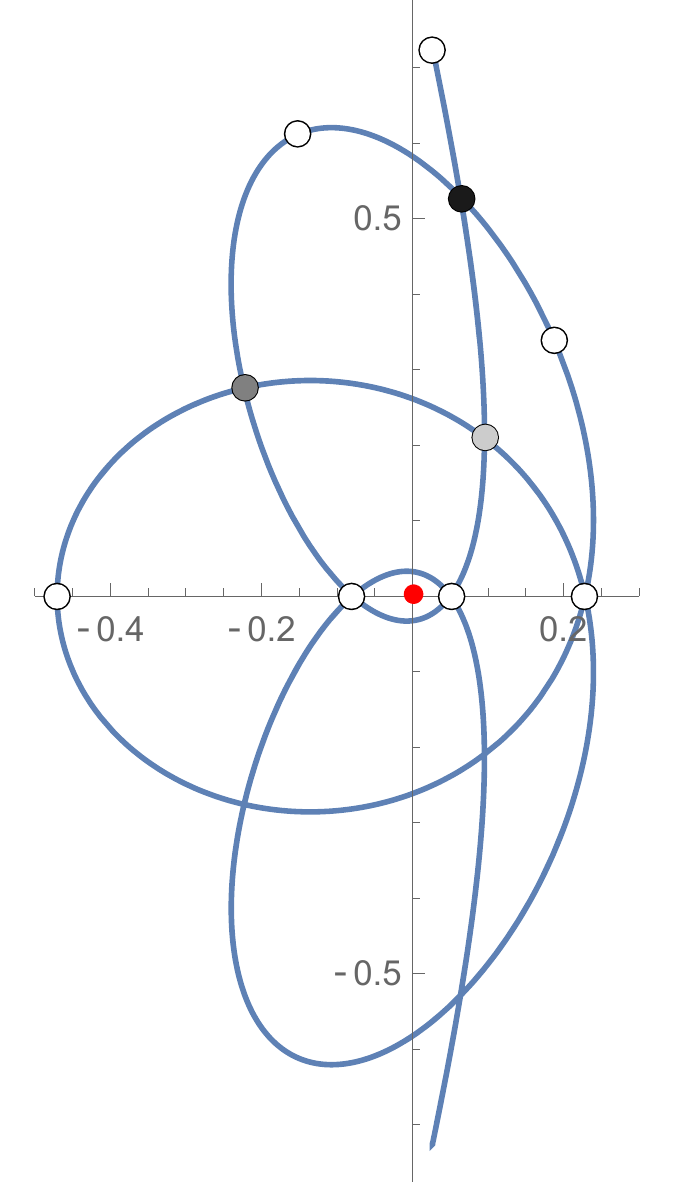}
\caption{ }
\label{lkugisdgs33}
\end{subfigure}
\caption{ The case $(k,l)=(4,3)$; (a) the blue curve is the graph of $\lambda=\lambda(t)$ and the red one is the graph of $\nu=\nu(t)$. The gray vertical lines divide the time interval $[0,T/4]$ by 12 subintervals of length $T/48$. The white dots represent the maximum or the minimum of the variables. The six gray dots make three pairs according to brightness which correspond to double points of the positive part;  (b)  a brake-brake orbit on a $T_{4,3}$-torus. Three gray dots on the positive part of the orbit correspond to the pairs described in (a).} 
\label{sd3fsdfsdfsdfsdd}
\end{figure}

In this example, following the proof of Proposition \ref{propoeeuloerodd} we study double points on a $T$-periodic brake-brake orbit $\gamma$ for $(k,l)=(4,3)$. Again by abuse of notation, we use the symbol $\gamma$ for the restriction $\gamma|_{[0, T/2]}$.

Abbreviate by $\gamma^+$ the positive part of $\gamma$. 
We mark $13$ points $t=jT/48$, $0 \leq j \leq 12$, on the interval $[0,T/4]$, see Figure \ref{figukluilu}. Note that
$$
\left \{ \frac{jT}{48} \; \bigg| \;j = 0,3,4,6,8,9,12  \right \}
$$
correspond to the maximum or  minimum of the variable $\lambda$ or $\nu$.  The first point is the braking point of $\gamma^+$ and the last point is the single point of $\gamma$ which lies on the $q_1$-axis. Among the other five points, $(4+3-1)/2 = 3$ points represent double points on the $q_1$-axis and the remaining two points are  single point of $\gamma$, see Figure \ref{lkugisdgs33}. 

  In view of the proof of Proposition \ref{propoeeuloerodd},  the six points $t=jT/48$, $j=1,2,5,7,10,11$, make  the three pairs 
$$
(1,7), (2,10), (5,11)
$$
which correspond to three double points of $\gamma^+$. Indeed, for example if we take $N=1$ in the proof of Claim 2 of the same proposition, then we have $m=n=P=Q=i=j=1$ and $a=b=0$. It follows that  $2jl+m = 2 \cdot 1 \cdot 3 + 1 = 7 = 2 \cdot 1 \cdot 4 \cdot -1  = 2ik-n$ frow which we conclude that $t=T/48$ and $t=7T/48$ represent a double point. Consequently, the brake-brake orbit  for $(k,l)=(4,3)$ has precisely nine quadruple points, see Figure \ref{lkugisdgs33}. 
\end{Example}

\begin{Theorem}\label{themeulerinariaht}   For the $T_{k,l}$-torus family in the $S$- or $S'$-region  in the Euler problem, we have
\[\mathcal{J}_1(T_{k,l}) = \begin{cases}  2kl-k-l+1      & \text{ if  $k+l$ is odd} \\  (kl-k-l+2)/2      & \text{ if  $k+l$ is even} \end{cases}
\]and
$$
\mathcal{J}_2(T_{k,l}) =   kl-k-l+1    .
$$
\end{Theorem}
 \begin{proof} In view of Propositions \ref{sdfkluliugliu3gsgd}   and \ref{propoeeuloerodd}, it remains to prove  the formula of $\mathcal{J}_2$ invariant for the case that $k+l$ is odd. 
 We only consider the case that $k$ is even and $l$ is odd. The other case can be proved in a similar way. 
 
Let $\gamma$ be a brake-brake orbit of period $T$.
As in the proof of Proposition \ref{propoeeuloerodd}, we only consider the half $\gamma|_{[0,T/2]}$, which will be denoted again by $\gamma$, and then  every self-intersection point is a double point.
Without loss of generality we may assume that $T=2$.
We also parametrize $K$ so that    the braking point with $q_2 >0$ is the initial point and hence the braking point with $q_2<0$ is the endpoint.

We assign each double point of $\gamma$ a rotation number as follows.
Let $p$ be a double point and we then find $0<t_0<t_1<1$ such that $\gamma(t_0) = \gamma(t_1) =p$.
The rotation number $\mathrm{rot}(p)$ is defined as the winding number of the tangent vector $\gamma'(t)$ for $t \in [t_0, t_1]$.
It is obvious that this number is an integer.

Recall that since the winding number of $\gamma$ is even, the pre-image $\widetilde{\gamma} := L^{-1}(\gamma)$ consists of two connected components $\gamma_1,\gamma_2$.
We fix a double point $p$ of $\gamma$  and $0<t_0< t_1<1$ as above.  
The pre-image $L^{-1}(p)$ is a pair of double points $p_1, p_2$ of $\widetilde{\gamma}$.
There are two possibilities: One is that $p_1,p_2$ are intersections between $\gamma_1,\gamma_2$, and the other is that $p_1,p_2$ is a double point of $\gamma_1,\gamma_2$, respectively. 
In order to determine the $\mathcal{J}_2$ invariant of $\gamma$, we have to count double points of $\gamma_1$ (or equivalently of $\gamma_2$) and hence we need the second scenario.
Note that since $L(z)= z^2$,  it is the case    if and only if $\mathrm{rot}(p)$ is even.

We claim that the number of double points of $\gamma_1$ is given by
\[
\frac{1}{2} ( k-1)(l-1)
\]  
from which the assertion of the theorem follows.
To this end, we take the approach of the proof of Proposition  \ref{propoeeuloerodd}.
 Suppose that $t_0 \in (0,1)$ (recall that we have assumed that $T=2$) represents a double point of $\gamma$ of even rotation number.
 We abbreviate $(\lambda_0 ,\nu_0) = (\lambda (t_0),\nu(t_0))$
 and  find that
 \[
 t = t_1, \frac{2}{k} + t_1, \frac{4}{k}+ t_1, \ldots, \frac{k-2}{k}+t_1
 \]
 at which $\lambda = \lambda_0$, where $t_1 \in (0, 1/2k)$ and that
 \[
 t=t_2, \frac{1}{l} \pm t_2, \frac{2}{l} \pm t_2,\ldots,  1- t_2
\]
at which $\nu = \nu_0$, where $t_2 \in (0, 1/2l )$.
Proceeding as in Claim 1 of the proof of Proposition \ref{propoeeuloerodd}, one can show that $t_0$ is of the form
\[
\frac{ \alpha k + 2 \beta l }{2 kl}, \quad \alpha, \beta \in \Z.
\]
 Note that
 \[
 \# \left\{ \frac{ \alpha k + 2 \beta l }{2 kl} \in [0,1] \relmiddle| \alpha, \beta \in \Z \right\} = kl+1.
 \]
As before, not every such point represents a double point.
We set
\[
A= \left\{ \frac{i}{k} \relmiddle| i =0,1,\ldots, k \right\}
\]
and
\[ 
B=\left\{ \frac{j}{2l} \relmiddle| j=0,1,\ldots, 2l \right\}.
\]
 The elements in $A$ correspond to the maximum of $\lvert \lambda\rvert $, and it is obvious that they do not represent double points. 
 Among $2l+1$ elements in $B$, the $(l+1)$ elements  $j/2l$, $j=0, 2, \ldots, 2l$,   correspond to the maximum of $\lvert \nu \rvert$, which do not represent double points.
 The remaining $l$ elements   $j/2l$, $j=1,3, \ldots, 2l-1$, correspond to $\nu =0$. 
 Here we have used, as in the proof of Proposition \ref{propoeeuloerodd},  the symbol $\nu$ for $\nu+\pi$ for sake of convenience.
Recall that $\nu=0$ represents the negative $q_1$-axis on which there exist a single point and $(l-1)/2$ double points.
Since $L(z) = z^2$ and $\gamma$ is $q_1$-symmetric, $\gamma_1$ is $q_2$-symmetric.
This in particular implies that a single point and the $(l-1)/2$ double points of $\gamma$ on the negative $q_1$-axis represent a single point and $(l-1)/2$ double points of $\gamma_1$ on the $q_2$-axis.
Proceeding as in Claim 2 of the proof of Proposition \ref{propoeeuloerodd}, we find the number of double points of $\gamma$ of even rotation number equals
\[
\frac{1}{2} \left( kl+1 - (k+1) - (l+1) +2\right) = \frac{1}{2}(k-1)(l-1).
\]
This proves the claim and completes the proof of the theorem. 
\end{proof}

\subsection{Comparison of the invariants of the two problems}  

Recall that each bounded component of  the regularized energy hypersurfaces of the rotating Kepler problem and the Euler problem is diffeomorphic to $\R P^3$. Since every even cover of a loop in $\R P^3$ lifts to a loop in $S^3$,  an even cover of any periodic orbit can be regarded as a knot in $S^3$. Note that if two knots $K_1$ and $K_2$ in $S^3$ are isotopic, then the projections $\pi(K_1)$ and $\pi(K_2)$ are also isotopic, where $\pi :S^3 \rightarrow \R P^3$. Therefore, two periodic orbits are never isotopic if their lifts in $S^3$ have different knot types.

 Let $\gamma^{\text{RKP}}$ be a $T_{k,l}$-type orbit in the rotating Kepler problem. We first suppose that   $k \pm l$ are even.  Recall  that the $T_{k,l}$-torus family bifurcates from an even-fold covered direct circular orbit which is contractible, see \cite[Section 7.2]{RKP}. Consequently,  $\gamma^{\text{RKP}}$ is contractible. It follows that the lift $\widetilde{\gamma}^{\text{RKP}}$ of $\gamma^{\text{RKP}}$ in $S^3$ consists of two components. It is obvious that they have the same knot type and hence without loss of generality, we may focus on one of two components.  If $k \pm l$ are odd, then by the same reasoning $\gamma^{\text{RKP}}$ is noncontractible. Therefore, by traversing $\gamma^{\text{RKP}}$ twice we lift it to  $\widetilde{\gamma}^{\text{RKP}}$ in $S^3$ which is a single   orbit.

Before determining knot types we note that if $k+l$ is even, then $\text{gcd}( (k+l)/2,  (k-l)/2)=1$, where $\text{gcd}(a,b)$ denotes the greatest common divisor of $a, b \in \R$.    Similarly, if $k+l$ is odd, then  we have $\text{gcd}(k+l,k-l)=1.$

\medskip

\textit{Case 1.} $\gamma^{\text{RKP}}$ is contractible.\\
Recall that the $T_{k,l}$-torus family bifurcates from the $(k-l)$-fold covered direct circular orbit and dies at the $(k+l)$-fold covered retrograde circular orbit. Since the $T_{k,l}$-torus family is a smooth two-parameter family of   $T_{k,l}$-type orbits, it is obvious that any two family members have the same knot type. Therefore, in order to determine the knot type of the $T_{k,l}$-torus family it suffices to consider a suitable representative. We choose two $T_{k,l}$-type orbits $\widetilde{\gamma}_1$ and $\widetilde{\gamma}_2$ which are sufficiently close to the lifts of the $(k-l)$-fold covered direct circular orbit and the $(k+l)$-fold covered retrograde circular orbit, respectively.  To explain them in more details,  we abbreviate by $(c_d, c_r)$ the interval  of  energies in which the $T_{k,l}$-torus family takes values, where  at $H_{\text{RKP}}=c_d$ and at  $H_{\text{RKP}}=c_r$, the $T_{k,l}$-torus family bifurcates and dies, respectively. The two orbits $\widetilde{\gamma}_1$ and $\widetilde{\gamma}_2$ are then given by $T_{k,l}$-type orbits having $H_{\text{RKP}}=c_d+\epsilon$ and  $H_{\text{RKP}}=c_r-\epsilon$, respectively, for $\epsilon >0$ small enough.  Consider the energy interval $[c_d,c_d+\epsilon]$ representing the solid torus in $S^3$ whose boundary is a $T_{k,l}$-torus containing  $\widetilde{\gamma}_1$ and whose core is the lift of the  $(k-l)$-fold covered direct circular orbit. Since $\epsilon>0$ is small enough, it follows that $\widetilde{\gamma}_1$ is a $( k-l , n)$-torus knot for some $n>0$ satisfying $\text{gcd}( k-l,n)=1$ .  In a similar way, we see that $\widetilde{\gamma}_2$ is a $( k+l , m)$-torus knot for some $m>0$ satisfying $\text{gcd}(k+l,m)=1$. Since $\widetilde{\gamma}_1$ and $\widetilde{\gamma}_2$ have the same knot type, it follows that $n=k+l$ and $m=k-l$. We conclude that the lift of any $T_{k,l}$-type orbit  a $((k+l)/2, (k-l)/2)$-torus knot.

\;\;

\textit{Case 2.} $\gamma$ is noncontractible.\\
In this case $k\pm l $ are odd. An argument similar with the one given in Case 1  shows that the lift of a $T_{k,l}$-type orbit  a $(k+l, k-l)$-torus knot. 

 \bigskip

Let $\alpha^{\text{Euler}}$ be a $T_{k,l}$-type orbit  in the Euler problem. We now  do a similar business for torus-type orbits in the $S$- and the $S'$-regions.  Recall from proposition \ref{lemmacoversing} and  \cite[Section 7.2]{RKP} that $T_{k,l}$-type orbits are contractible if $k+l$ is odd and noncontractible if $k+l$ is even. An argument similar with the one given above   with two $T_{k,l}$-type orbits, which are sufficiently close to (multiple covered) exterior and interior collision orbits, shows that regardless of the parity of $k+l$, the lift of  $\alpha^{\text{Euler}}$ is a $(k,l)$-torus knot.

\medskip

We have proven
\begin{Proposition}\label{sdlfisholisef} Let $\gamma^{\text{RKP}}$ and $\alpha^{\text{Euler}}$ be  $T_{k,l}$-type orbits in the rotating Kepler problem and in the Euler problem, respectively. Then
 \begin{enumerate}[label=(\roman*)]
\item $\gamma^{\text{RKP}}$ is either a $(k+l,k-l)$-torus knot or a $( (k+l)/2, (k-l)/2)$-torus knot if $k+l$ is odd or if $k+l$ is even, respectively;
\item $\alpha^{\text{Euler}}$  is a $(k,l)$-torus knot. 
\end{enumerate}
\end{Proposition}

\medskip

We distinguish the following two cases:

\medskip

\noindent \textit{{\large Case 1. $k+l$ is even.}}\\
In view of   Proposition \ref{sdlfisholisef}  the lifts of   $T_{k,l}$-type orbits in the rotating Kepler problem and of $T_{(k+l)/2,(k-l)/2}$-type orbits in the Euler problem  have the same knot type and their projections on $\R P^3$ are \textit{contractible}. Recall from Theorems \ref{maintheoremRKP}    and  \ref{themeulerinariaht} that
$$
\mathcal{J}_1 (T_{k,l}^{\text{RKP}})=   1 - k + \frac{k^2}{2} - \frac{l^2}{2}  
$$
and
\begin{equation}\label{eqeulkerliussdg}
 \mathcal{J}_1(T_{r,s}^{\text{Euler}})=  2rs-r-s+1 \hspace{2mm}\text{ if $r+s$ is odd.}
\end{equation}
Plugging $r=(k+l)/2$ and $s=(k-l)/2$ into \eqref{eqeulkerliussdg}  gives rise to 
$$
 \mathcal{J}_1(T_{(k+l)/2,(k-l)/2}^{\text{Euler}}) = \frac{k^2-l^2}{2} -k+1=\mathcal{J}_1 (T_{k,l}^{\text{RKP}}).
$$
In a similar way, we find that
 \[
 \mathcal{J}_2 ( T^{\text{Euler}}_{ (k+l)/2, (k-l)/2} ) = 1 - k + \frac{k^2 - l^2}{4} =   \mathcal{J}_2 ( T^{\text{RKP}}_{k,l}).
 \]

\;\;

\noindent \textit{{\large Case 2. $k+l$ is odd. }}\\
In this case,     $T_{k,l}$-type orbits in the rotating Kepler problem and $T_{k+l,k-l}$-type orbits in the Euler problem have the same knot type and their projections are \textit{noncontractible}. Replacing $(k,l)$ by  $(k+l,k-l)$ as before, we find that
$$
 \mathcal{J}_1(T_{k+l,k-l}^{\text{Euler}}) = \frac{k^2-l^2}{2} -k+1=\mathcal{J}_1 (T_{k,l}^{\text{RKP}}).
$$
Since the $\mathcal{J}_2$ invariant is determined by  the $\mathcal{J}_1$ invariant, see Proposition \ref{theoremsdih4}, we also obtain that 
 $$
 \mathcal{J}_2(T_{k+l,k-l}^{\text{Euler}}) = (k-1)^2-l^2=\mathcal{J}_2 (T_{k,l}^{\text{RKP}}).
$$

\;\;\;

We have proven

\begin{Theorem}\label{mainsdliufhsdg}  Let $\gamma^{\text{RKP}}$ and $\alpha^{\text{Euler}}$ be torus-type orbit in the rotating Kepler problem and in the Euler problem, respectively, of same knot type. They have the same $\mathcal{J}_1$ and $\mathcal{J}_2$ invariants.
\end{Theorem}

Let $\gamma^{\text{RKP}}$ and $\alpha^{\text{Euler}}$ be torus-type orbits in the rotating Kepler problem and in the Euler problem with a given $\mu = \mu_0$, respectively. We assume that they are continued by families of periodic orbits to periodic orbits in the PCR3BP with $\mu=\mu_0$, say  $\gamma^{\text{3BP}}$ and $\alpha^{\text{3BP}}$, respectively. 

\begin{Remark}  \rm If $\mu_0$ is sufficiently small, then the aforementioned family from $\gamma^{\text{RKP}}$ to $\gamma^{\text{3BP}}$ indeed exists, see \cite{Arenstorf, Barrar}.
\end{Remark}

Since $\mathcal{J}_1$ and $\mathcal{J}_2$ are invariants for families of periodic orbits in Stark-Zeeman systems, we see that $ \mathcal{J}_i ( \gamma^{\text{3BP}}) = \mathcal{J}_i ( \gamma^{\text{RKP}})$ and $ \mathcal{J}_i ( \alpha^{\text{3BP}}) = \mathcal{J}_i ( \alpha^{\text{Euler}})$, $i=1,2$.

\;\;\;

\end{document}